\documentclass[sn-mathphys,Numbered]{sn-jnl}

\usepackage{graphicx}%
\usepackage{multirow}%
\usepackage{amsthm}%
\usepackage{amsmath,amssymb,amsfonts}%

\usepackage{mathtools}
\usepackage{todonotes}
\usepackage{mathrsfs}%
\usepackage[title]{appendix}%
\usepackage{xcolor}%
\usepackage{textcomp}%
\usepackage{xfrac}
\usepackage{manyfoot}%
\usepackage{booktabs}%
\usepackage{listings}%
\usepackage{pgfplots}
\pgfplotsset{compat=1.18}
\usepackage{multicol}
\pgfplotsset{every axis/.append style={
                    label style={font=\small},
                    tick label style={font=\small}  
                    }}

\usepackage[]{xcolor}

\definecolor{col1}{rgb}{     0, 0.4470, 0.7410}
\definecolor{col2}{rgb}{0.8500, 0.3250, 0.0980}
\definecolor{col3}{rgb}{0.9290, 0.6940, 0.1250}
\definecolor{col4}{rgb}{0.4940, 0.1840, 0.5560}

\renewcommand{\:}{\mathrel{\coloneqq}}
\usepackage{hyperref}
\usepackage[nameinlink,capitalize]{cleveref}

\newcommand{\plotfont}{\fontsize{8}{10}\selectfont} %<-- New fontsize

\newcommand{\R}{\mathbb{R}}

\renewcommand{\epsilon}{\varepsilon}
\newcommand{\N}{\mathbb N}

\newcommand{\eps}{\epsilon}

\newcommand{\sL}{{\mathsf{L}\!}}
\newcommand{\sBV}{{\mathsf{BV}}}
\newcommand{\sW}{\mathsf W}
\newcommand{\sC}{\mathsf C}
\newcommand{\sTV}{\mathsf{TV}}

\newcommand{\mD}{\mathcal D}
  \newcommand{\sgn}{\ensuremath{\textnormal{sgn}}}
\newcommand{\e}{\mathrm e}

\newcommand{\OT}{{\Omega_{T}}}
\newcommand{\loc}{\textnormal{loc}}

\newcommand{\dd}{\ensuremath{\,\mathrm{d}}}
\DeclareMathOperator{\supp}{supp}

\newtheoremstyle{thmstyleone}% Numbered
    {18pt plus2pt minus1pt}% Space above
    {18pt plus2pt minus1pt}% Space below
    {\small\itshape}% Body font
    {0pt}% Indent amount
    {\small\bfseries}% Theorem head font
    {}% Punctuation after theorem head
    {.5em}% Space after theorem headi
    {}%

\newtheoremstyle{thmstyletwo}% Numbered
    {18pt plus2pt minus1pt}% Space above
    {18pt plus2pt minus1pt}% Space below
    {\small\normalfont}% Body font
    {0pt}% Indent amount
    {\small\itshape}% Theorem head font
    {}% Punctuation after theorem head
    {.5em}% Space after theorem headi
    {}% Theorem head spec (can be left empty, meaning `normal')
\newtheoremstyle{thmstylethree}% Definition
    {18pt plus2pt minus1pt}% Space above
    {18pt plus2pt minus1pt}% Space below
    {\small\normalfont}% Body font
    {0pt}% Indent amount
    {\small\bfseries}% Theorem head font
    {}% Punctuation after theorem head
    {.5em}% Space after theorem headi
    {}

\theoremstyle{thmstyleone}%
\newtheorem{theorem}{Theorem}%  meant for continuous numbers
\newtheorem{proposition}[theorem]{Proposition}% 

\theoremstyle{thmstyletwo}%
\theoremstyle{thmstylethree}%
\newtheorem{remark}{Remark}%
\newtheorem{assumption}{Assumption}

\newtheorem{definition}{Definition}%
\newtheorem{lemma}{Lemma}

\crefname{remark}{Rem.}{Remarks}
\crefname{lemma}{Lem.}{Lemmata}
\crefname{proposition}{Prop.}{Propositions}
\crefname{corollary}{Cor.}{Corallaries}
\crefname{theorem}{Thm.}{Theorems}
\crefname{assumption}{Asm.}{Assumptions}
\crefname{definition}{Defn.}{Definitions}
\crefname{example}{Ex.}{Examples}
\crefname{section}{Sec.}{Sections}

\begin{document}

\title[Singular Limit Problem for nonlocal conservation laws: Kernels with fixed support]{On the singular limit problem for nonlocal conservation laws: A general approximation result for kernels with fixed support}

\author*[1]{\fnm{Alexander} \sur{Keimer}}\email{alexander.keimer@fau.de}

\author*[1,2]{\fnm{Lukas} \sur{Pflug}}\email{lukas.pflug@fau.de}

\affil*[1]{\orgdiv{Department of Mathematics}, \orgname{Friedrich-Alexander-University Erlangen-Nürnberg (FAU)}, \orgaddress{\street{Cauerstraße 11}, \city{Erlangen}, \postcode{91058}, \country{Germany}}}

\affil[2]{\orgdiv{Competence Center Scientific Computing (CSC)}, \orgname{Friedrich-Alexander-University Erlangen-Nürnberg (FAU)}, \orgaddress{\street{Martensstraße 5a}, \city{Erlangen}, \postcode{91058}, \country{Germany}}}

\abstract{We prove the convergence of solutions of nonlocal conservation laws to their local entropic counterpart for a fundamentally extended class of nonlocal kernels when these kernels approach a Dirac distribution.
The nonlocal kernels are assumed to have fixed support and do not have to be monotonic. With sharp estimates of the nonlocal kernels and a surrogate nonlocal quantity, we prove compactness in \(\sC(\sL^{1}_{\text{loc}})\) which allow passing to the limit in the weak formulation. A careful analysis of the entropy condition of local conservation laws together with the named estimators for the considered kernels enable it to prove the entropy admissibility of the nonlocal equation in the limit, completing the convergence proof.  
}



\maketitle

\section{Introduction}\label{sec:introduction}
In this contribution, we focus on the convergence of solutions of nonlocal conservation laws (for details see \cite{scialanga,pflug,crippa2013existence,kloeden,chiarello,friedrich2018godunov}) to the corresponding local conservation laws when the nonlocal kernel converges to a Dirac distribution.
This problem has been studied quite intensively over the last decade. \cite{teixeira} observed for the first time that, for the numerical approximation, such a convergence might hold.
Later, for rather general nonlocal conservation laws with an additional monotonicity assumption on the datum a general convergence result was proven \cite{pflug4}. Several other papers showed such convergence under additional assumptions on velocity and/or initial datum \cite{Crippa2021,spinolo,bressan2019traffic,bressan2021entropy}, until a general approximation result relying on the exponential kernel recently became available \cite{coclite2022general}. The general idea was to consider, instead of the solution, the nonlocal approximation and to derive another PDE in purely nonlocal terms, allowing us to obtain uniform \(\sTV\) bounds on the nonlocal term --- enough for passing to the limit in the weak formulation. Combining \cite{bressan2021entropy} together with a classical compactness result then gave the claimed convergence in \cite{colombo2023nonlocal}, and then generalized to a rather broad class of convex kernels (for further details, see the later comparison to the results in this manuscript, particularly \cref{subsec:comparison}) and also to nonlocal conservation laws where the nonlocal operator acts on the velocity rather than on the density \cite{friedrich2023conservation}.

Let us briefly dwell on why this convergence is of such importance:
\begin{itemize}
    \item It ties together local and nonlocal conservation laws so that each local modelling can be understood as a limit of nonlocal modelling. Even more, as most local equations can be interpreted as limits of nonlocal equations (at least on the particle level), we thus have a rigorous proof of such limiting models.
    \item It enables weak (entropy) solutions of local conservation laws to be defined by means of limits or nonlocal conservation laws with the advantage that, for nonlocal conservation laws, the hyperbolicity of the dynamics is conserved.
    \item With the previous points in mind, one can approximate, for instance, local (optimal) control problems and stabilization problems with their nonlocal counterparts where the underlying theory---weak solutions are unique---is easier. Moreover, the non-dif\-fe\-ren\-tiability for entropy solutions of (local) conservation laws (see e.g.\ \cite{hajian2019total,herty2023numerics,Bouchut1999,pfaff2015optimal,ulbrich2002sensitivity}) can be overcome by using nonlocal approximations and passing to the limit afterwards.
    \item The result can be seen as a first step for nonlocal approximation results in the system of conservation law as well as the multi-$d$ case.
\end{itemize}

We will show in this manuscript that this convergence holds for a hitherto unstudied and broad class of nonlocal kernels which we call ``kernels with fixed support.'' These kernels do not even need to obey the classical monotonicity assumption which is particular reasonable in traffic flow modelling: The further one is from the current position, the less the nonlocal contribution should be. For this class of kernels, the nonlocal scaling is not in the kernel support but in the exponent of the kernel and in the exponential kernel being the kernel which fits into both categories. 

The crucial step in establishing the convergence is to introduce a surrogate quantity, the exponential weight of the solution, which allows us to express the dynamics entirely in the nonlocal operator for which we can show compactness in \(\sC(\sL^{1}_{\text{loc}})\), taking advantage of estimators for the kernels considered.

Let us consider the following problem on \((t,x)\in(0,T)\times\R \coloneqq \OT\)
\begin{align}
    q_{t}+\big(V(W[q,\gamma_\eta](t,x))q(t,x)\big)_{x}&=0,&& (t,x)\in\OT \notag\\
    W[q,\gamma_{\eta}](t,x)&=\int_{x}^{\infty} \gamma_\eta(y-x)q(t,y)\dd y, && (t,x)\in\OT\label{eqn:NBL}\\
    q(0,x)&=q_{0}(x),&& x\in\R \notag
\end{align}
with the integral kernel $\gamma_\eta$ defined as
\begin{align}
   \gamma_\eta(x) \: c_\eta \gamma(x)^\frac{1}{\eta} \text{ with } c_\eta \: \bigg(\int_{0}^{\infty} \gamma(y)^\frac{1}{\eta}\dd y\bigg)^{-1},\; \forall (x,\eta) \in \R_{\geq 0}\times \R_{>0}.
    \label{eqn:cgamma}
\end{align}
For kernels $\gamma_\eta$ defined as above, with a \(\sTV\) bounded $\gamma$, with a negative first derivative at zero, bounded second derivative in a neighborhood of zero, i.e. $\gamma|_{(0,\delta)} \in W^{2,\infty}((0,\delta))$ for some $\delta \in \R_{>0}$, and some weak monotonicity principle, viz., $\gamma(x) \geq \gamma(y)$ for all $x\in (0,\delta)$ and all $y\in \R_{>x}$, we will show that the related solutions $q_\eta$ converge strongly in \(\sC(\sL^1)\) to the entropy admissible solution of the corresponding local conservation law.

\subsection{Comparison with the existing convergence results}\label{subsec:comparison}
The previously mentioned convergence results rely strongly on a spatial scaling of the involved kernels, i.e., for a given kernel $\gamma \in \sL^1(\R_{>0};\R_{\geq 0})$, the limit of the following scaled kernels were considered:
\begin{align}
    \gamma_{\eta}(x)\coloneqq \tfrac{1}{\eta}\gamma \big(\tfrac{x}{\eta}\big),\ x\in\R
\end{align}
while we consider a scaling in the power as introduced in \cref{eqn:cgamma}. In \cref{fig:kernel_comparison}, the two types of scaling are illustrated for two different kernels. Roughly speaking, the power scaling ``reduces'' discontinuities, leading to ``smoother'' kernels compared to the spatial scaling where the discontinuities get up-scaled by \(\eta^{-1}\).

The exponential kernels, for which the first general proof of convergence was obtained in \cite{coclite2022general}, can be interpreted as either of  spatial scaling or power scaling, as can be seen in the following calculation:
\begin{align}
    \tfrac{1}{\eta}\exp\big(-\tfrac{x}{\eta}\big) =  \tfrac{1}{\eta}\exp(-x)^\frac{1}{\eta} ,\ x\in\R.
    \label{eqn:gamma_x_y}
\end{align}
To compare the convergence results obtained in the literature (\cite{coclite2022general,colombo2023nonlocal}) with the results in this contribution, we separate the scaling type from the underlying kernel $\gamma : \R_{\geq 0} \rightarrow \R_{\geq 0}$, i.e.\ we choose a \(\gamma\) as kernel and detail for which scaling the convergence is or has been proven.
In the most general result \cite{colombo2023nonlocal}, the authors require the kernel to be (following the notation in this manuscript)
\[
\sL^{1}\cap\sL^{\infty}\big(\R_{\geq0};\R_{\geq0}\big)\ \wedge\ \text{convex and monotonically decreasing.}
\]
However, this implies in particular continuity of the kernel on \(\R_{\geq 0}\) (after selecting a proper representative).
The kernel class considered in this contribution can, however, be discontinuous outside an arbitrarily small neighborhood around zero (where we indeed require higher regularity), see \cref{ass:gamma}, and the kernel can also be non-monotone outside an even smaller environment so that the values of \(\gamma\) inside this smaller neighborhood dominate the remaining values of \(\gamma\). This brings us to the statement that the kernel class considered here with the power scaling is much broader than all existing results on convergence in the literature (but recognizing that the higher regularity in the small neighborhood around zero is not required in \cite{colombo2023nonlocal}).

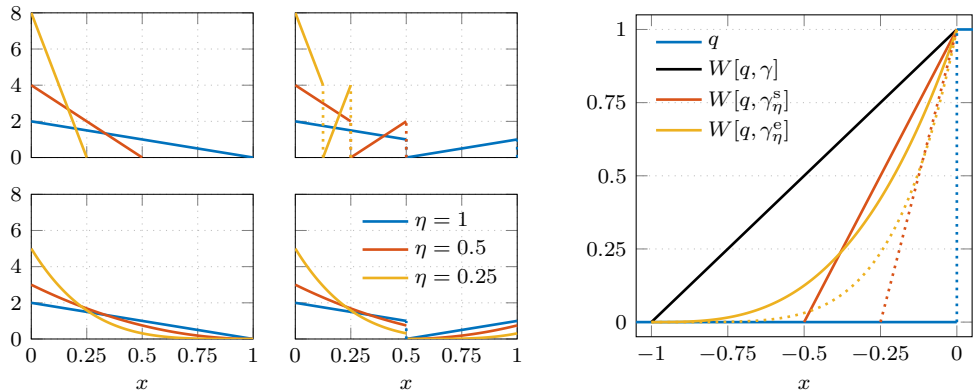
\begin{figure}
\centering
\begin{minipage}[b]{.57\textwidth}
   \begin{tikzpicture}
\begin{axis}[grid=both,
height=3.5cm,width=4.5cm,xscale=1,
            grid style=dotted,
          xmin=0,ymin=0,
          xmax=1,ymax=8,
          restrict y to domain=0:10,
          restrict x to domain=0:10,
          xtick={0,0.25,0.5,0.75,1},
          legend style={draw=none},
          axis line style={line width=0.25pt},xticklabel={\hphantom{1}},tick label style={font=\plotfont}, label style={font=\plotfont}]
\addplot[col1,domain=0:1,samples=100,line width=1pt]  {2*(1-x)};
\addplot[col2,domain=0:0.5,samples=100,line width=1pt]  {4*(1-2*x)};
\addplot[col3,domain=0:0.25,samples=100,line width=1pt]  {8*(1-4*x)};
\end{axis}
\end{tikzpicture}
\begin{tikzpicture}
\begin{axis}[grid=both,
height=3.5cm,width=4.5cm,xscale=1,
          grid style=dotted,
          xmin=0,ymin=0,
          xmax=1,ymax=8,
          restrict y to domain=0:10,
          restrict x to domain=0:10,
          legend style={draw=none},
          xtick={0,0.25,0.5,0.75,1},
          axis line style={line width=0.25pt},xticklabel={\hphantom{0}},yticklabel=\empty,tick label style={font=\plotfont}, label style={font=\plotfont}]
\addplot[col1,domain=0:0.5,samples=100,line width=1pt]  {2-2*x};
\addplot[col2,domain=0:0.25,samples=100,line width=1pt]  {4-8*x};
\addplot[col3,domain=0:0.125,samples=100,line width=1pt]  {8-32*x};

\addplot[col1,domain=0.5:1,samples=100,line width=1pt]  {2*x-1};
\addplot[col2,domain=0.25:0.5,samples=100,line width=1pt]  {8*x-2};
\addplot[col3,domain=0.125:0.25,samples=100,line width=1pt]  {32*x-4};
\addplot[col1,domain=0.5:1,samples=100,line width=1pt,dotted]coordinates  {(0.5,0)(0.5,1)};
\addplot[col2,domain=0.5:1,samples=100,line width=1pt,dotted]coordinates  {(0.25,0)(0.25,2)};
\addplot[col3,domain=0.5:1,samples=100,line width=1pt,dotted]coordinates  {(0.125,0)(0.125,4)};
\addplot[col1,domain=0.5:1,samples=100,line width=1pt,dotted]coordinates  {(1,0)(1,1)};
\addplot[col2,domain=0.5:1,samples=100,line width=1pt,dotted]coordinates  {(0.5,0)(0.5,2)};
\addplot[col3,domain=0.5:1,samples=100,line width=1pt,dotted]coordinates  {(0.25,0)(0.25,4)};

\end{axis}
\end{tikzpicture}

\begin{tikzpicture}
\begin{axis}[grid=both,
height=3.5cm,width=4.5cm,xscale=1,
          grid style=dotted,
          xmin=0,ymin=0,
          xmax=1,ymax=8,
          xlabel={$x$},
          restrict y to domain=0:10,
          restrict x to domain=0:10,
          xtick={0,0.25,0.5,0.75,1},
          legend style={draw=none},
          axis line style={line width=0.25pt},tick label style={font=\plotfont}, label style={font=\plotfont}]
\addplot[col1,domain=0:1,samples=100,line width=1pt]  {2*(1-x)};
\addplot[col2,domain=0:1,samples=100,line width=1pt]  {3*pow((1-x),2)};
\addplot[col3,domain=0:1,samples=100,line width=1pt]  {5*pow((1-x),4)};
\end{axis}
\end{tikzpicture}
\begin{tikzpicture}
\begin{axis}[grid=both,
height=3.5cm,width=4.5cm,xscale=1,
          grid style=dotted,
          xmin=0,ymin=0,
          xmax=1,ymax=8,
          xlabel={$x$},
          restrict y to domain=0:10,
          restrict x to domain=0:10,
          xtick={0,0.25,0.5,0.75,1},
          legend style={draw=none,at={(0.975,0.95)},fill=none,font=\plotfont},
          axis line style={line width=0.25pt},yticklabels=\empty,legend cell align={left},tick label style={font=\plotfont}, label style={font=\plotfont}]
\addplot[col1,domain=0:0.5,samples=100,line width=1pt]  {2-2*x};
\addplot[col2,domain=0:0.5,samples=100,line width=1pt]  {3/4*pow(2-2*x,2)};
\addplot[col3,domain=0:0.5,samples=100,line width=1pt]  {5/16*pow(2-2*x,4)};

\addplot[col1,domain=0.5:1,samples=100,line width=1pt]  {2*x-1};
\addplot[col2,domain=0.5:1,samples=100,line width=1pt]  {3/4*pow(2*x-1,2)};
\addplot[col3,domain=0.5:1,samples=100,line width=1pt]  {5/16*pow(2*x-1,4)};
\addplot[col1,domain=0.5:1,samples=100,line width=1pt,dotted]coordinates  {(0.5,0)(0.5,1)};

\legend{$\eta=1$, $\eta=0.5$, $\eta=0.25$}
\end{axis}
\end{tikzpicture}
\end{minipage}
\begin{minipage}[b]{.42\textwidth}
\begin{tikzpicture}
\begin{axis}[ymajorgrids,
grid style=dotted,height=5.85cm,width=6cm,
          xmin=-1.05,ymin=-0.05,
          xmax=0.05,ymax=1.05,
          restrict y to domain=0:1,
          restrict x to domain=-2:1,
          xtick={-1.0,-0.75,-0.5,-0.25,0.0},
          ytick={0,0.25,0.5,0.75,1},xlabel={$x$},
          legend style={draw=none,fill=none,font=\plotfont},legend pos=north west,legend cell align={left},tick label style={font=\plotfont}, label style={font=\plotfont}]
\addplot[col1,domain=-2:0,samples=100,line width=1pt]  {0};
\addplot[black,domain=-1:0,samples=100,line width=1pt]  {(x+1)};
\addplot[col2,domain=-1/2:0,samples=100,line width=1pt]  {2*(x+1/2)};
\addplot[col3,domain=-1:0,samples=100,line width=1pt]  {pow(1+x,3)};

\addplot[col1,domain=0:1,samples=100,line width=1pt]  {1};
\addplot[col2,domain=-1/4:0,samples=100,line width=1pt,dotted]  {4*(x+1/4)};
\addplot[col3,domain=-1:0,samples=100,line width=1pt,dotted]  {pow(1+x,5)};
\addplot[col1,domain=-2:0,samples=100,line width=1pt,dotted]coordinates  {(0,0)(0,1)};
\legend{\(q\),\({W[q,\gamma]}\),\({W[q,\gamma_\eta^\text{s}]}\),\({W[q,\gamma_\eta^\text{e}]}\)}
\end{axis}
\end{tikzpicture}
\end{minipage}
 \caption{\textbf{Left:} In the \textbf{top} row, the spatial scaling, in the \textbf{bottom} row, the power scaling is visualized for two different values of $\gamma$, viz. a linear kernel ($\gamma \equiv 2(1-\cdot)$, \textbf{left}) and a non-monotone, piecewise-linear kernel ($\gamma \equiv 2(1-\cdot)  \chi_{(0,0.5)}(\cdot) + (2\cdot-1)\chi_{(0.5,1)}(\cdot)$, \textbf{right}). In all cases, the kernels $\gamma_\eta$ for $\eta \in \{1,0.5,0.25\}$ are shown; \textbf{Right:} For the given \(q\equiv \chi_{\R_{>0}}\) we visualize the smoothing by the nonlocal operator. For this, we compare the smoothing of the spatial scaling, i.e., $\gamma_{\eta}^{\text{s}} \equiv \sfrac{1}{\eta}\gamma\big(\sfrac{\cdot}{\eta}\big)$ with the power scaling, i.e.,  $\gamma_{\eta}^{\text{e}} \equiv c_\eta \gamma^{\sfrac{1}{\eta}}$. We chose $\eta = 1$ where both scalings coincide, i.e., $\gamma_1^\text{s} \equiv \gamma_{1}^{\text{e}}$ (black solid line),  $\eta = \sfrac{1}{2}$ (solid colored lines) and
 $\eta = \sfrac{1}{4}$ (dashed colored lines).
 }
\label{fig:kernel_comparison}
\end{figure}

\subsection{Outline}
In \cref{sec:introduction}, we introduced the problem setup and have demonstrated in \cref{subsec:comparison} how the presented result is a generalization upon the existing literature.
Next, \cref{sec:problem_statement} will present the precise statement of the dynamics under consideration as well as the fundamental assumptions on every datum, starting with the nonlocal dynamics in \cref{subsec:nonlocal}, continuing with the assumptions in \cref{subsec:assumptions} and \cref{subsec:local_dynamics} with the corresponding local conservation laws.

In \cref{sec:existence}, we tackle the existence and uniqueness of the nonlocal dynamics. Although there are plenty results available, our choice of more general kernels requires us to revisit this theory and obtain global existence based on a weakened form of maximum principle.

\Cref{sec:kernel} is dedicated to estimating the kernel class under consideration, mainly the prefactor \(c_\eta\) as introduced in \cref{eqn:cgamma}, in the nonlocal approximation---which will later turn out to be crucial.  In \cref{sec:TV_uniform}, a nonlocal surrogate term of exponential type is introduced for which uniform \(\sTV\) estimates are shown, leading to compactness in \(\sL^{1}\) and, additionally, to strong convergence of nonlocal solution to local weak solutions. 

Having obtained this convergence, one needs to show that the limit is also entropy admissible, which is demonstrated in \cref{sec:convergence}. Eventually, \cref{sec:open_problems} concludes this paper by highlighting open questions and future research directions.

\section{Precise problem statement and assumptions}\label{sec:problem_statement}
In this section, we start with a precise definition of the nonlocal conservation laws considered, the regularity and properties required on the nonlocal kernel, and we will introduce the solution concepts---entropy solutions with entropy-flux pairs---for the corresponding local conservation law.

\subsection{Nonlocal dynamics}\label{subsec:nonlocal}

\begin{definition}[Nonlocal conservation law]\label{defi:nonlocal_conservation_law}
 For a nonlocal kernel \(\gamma \in \sL^1(\R_{\geq 0};\R_{\geq 0})\), \(q_{0}\in \sL^\infty(\R)\) and \(T\in \R_{>0}\) the nonlocal conservation law considered in this contribution is given by the following Cauchy problem in the density \(q:\OT\rightarrow\R\):
 \begin{align}
\begin{split}
    q_{t}+\big(V(W[q,\gamma](t,x))q(t,x)\big)_{x}&=0,  \\
    q(0,x)&=q_{0}(x),  
\end{split}\label{eqn:nbl} && (t,x)\in\OT,\\
\intertext{with the nonlocal term defined as} 
        W[q,\gamma](t,x)&\coloneqq\int_{x}^{\infty} \gamma(y-x)q(t,y)\dd y, && (t,x)\in\OT.
        \label{eqn:nonlocalterm}
\end{align}
\end{definition}

\begin{assumption}[Initial datum and velocity]\label{ass:initialdatum_velocity} We assume the following
\begin{description}
    \item[Initial datum:] \(q_{0}\in \sBV_\text{loc}(\R;\R_{\geq0})\) with $|q_0|_{\sTV(\R)}< \infty$ and \( \|\cdot\|_{\sBV(\R)}\coloneqq\|\cdot\|_{\sL^{1}(\R)}+|\cdot|_{\sTV(\R)}\)
    \item[Velocity:] \(V\in\sW^{1,\infty}_{\loc}(\R): V'\leqq0 \text{ and }\R \ni s \mapsto sV(s) \text{ strict concave}\)
\end{description}
\end{assumption}
Since the assumptions on the involved kernels are of crucial importance in the following, we will introduce and discuss them separately in the following subsection.

\subsection{Assumptions on the involved kernel}\label{subsec:assumptions}
This subsection introduces the class of nonlocal kernels considered in this work.
\begin{assumption}[Assumptions on the integral kernel] \label{ass:gamma}
For the integration kernel $\gamma$ as in \cref{eqn:cgamma} we assume that $\exists \delta \in \R_{>0}$ s.t:
\begin{description}
    \item[1) Integrability and total variation bound:] \(\gamma \in  \sBV(\R_{>0};\R_{\geqq 0})\)
    \item[2) Bounded second derivative in arbitrary small neighborhood:] \(\gamma|_{(0,\delta)} \in \sW^{2,\infty}((0,\delta))\)
    \item[3) Negative derivative in zero:] \(\gamma'(0) < 0\)
    \item[4) Upper bound on \(\R_{>\delta}\):] \(\gamma(x) \geq \gamma(y) \ \forall (x,y) \in (0,\delta) \times \R_{>\delta} \)
\end{description}
Without loss of generality, we assume, in addition,
\begin{description}
    \item[5) Strict monotonicity on \((0,\delta)\):] \(\gamma'|_{(0,\delta)} < 0\)
    \item[6) Normalization:] \(\gamma(0) = 1.\)
\end{description}
\end{assumption}

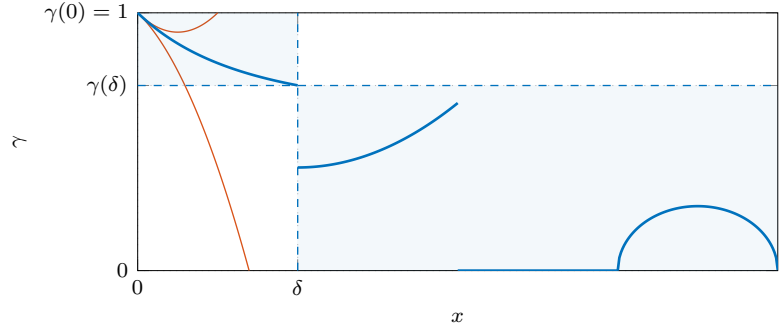
\begin{figure}
\centering
\begin{tikzpicture}
\begin{axis}[height=5cm,width=10cm,xmin=0,ymin=0,xmax=2,ymax=1,grid=both,grid style=dotted,restrict y to domain=0:10,restrict x to domain=0:10,axis line style={line width=0.25pt},xtick={0,0.5},
xticklabels={$0$,$\delta$},xlabel={$x$},ylabel={$\gamma$},ytick={0,0.718,1},yticklabels={$0$,$\gamma(\delta)$,$\gamma(0)=1$},tick label style={font=\plotfont}, label style={font=\plotfont}]
\addplot[col2,domain=0:0.25,samples=100,line width=.5pt,style=solid]  {1-1.2*x+4.8*x^2};
\addplot[col2,domain=0:0.348242,samples=100,line width=.5pt,style=solid]  {1-1.2*x-4.8*x^2};
\addplot[col1,domain=0:0.5,samples=100,line width=1pt]  {0.6+0.4/(1+x)^3};
\addplot[col1,domain=0.5:1,samples=100,line width=1pt]  {(x-0.5)^2+0.4};
\addplot[col1,domain=1:1.5,samples=100,line width=1pt]  {0};
\addplot[col1,domain=1.5:2,samples=100,line width=1pt]  {sqrt(x-1.5)*sqrt(2-x)};
\addplot[line width=.5pt, samples=50, smooth,domain=0:6,col1,style=dashed] coordinates {(0.5,0)(0.5,1)};
\addplot[line width=.5pt, samples=50, smooth,domain=0:6,col1,style=dashed] coordinates {(0,0.718)(2.0,0.718)};
\addplot [draw=none,color=blue,mark=none,fill=col1,fill opacity=0.05]coordinates {(0, 0.718) (0.5, 0.718)(0.5, 1)(0,1)};
\addplot [draw=none,color=blue,mark=none,fill=col1,fill opacity=0.05]coordinates {(0.5, 0) (0.5, 0.718)(2, 0.718)(2,0)};
\end{axis}
\end{tikzpicture}
\caption{Visualization of \cref{ass:gamma}, i.e., the assumption on the involved kernel functions. The bounded second derivative on \([0,\delta)\) for an arbitrary small \(\delta \in \R_{>0}\) and the negative derivative at zero are sketched by the \textcolor{col2}{red} quadratic over-/under-estimator (Item 2 and 3) which also play a crucial role in the estimates performed in \cref{sec:kernel}. The upper bound on \(\R_{>\delta}\) is visualized by the dashed \textcolor{col1}{blue} lines. In the top left \textcolor{col1}{blueish} area, we assume \(\gamma \in W^{2,\infty}\); in the lower right \textcolor{col1}{blueish} marked area, we only need \(\gamma \in \sBV\).}
\label{fig:ass_gamma}
\end{figure}
The first assumption, in particular $\gamma \in \sL^1(\R)$, is crucial to have the nonlocal term well-posed for $\eta \leq 1$ recalling that we scale the kernel with the exponent \(\eta^{-1}\). The second is necessary to have the solution $q_\eta$ satisfies a weak kind of maximum principle as provided in \cref{theo:max}.
The third assumption is not very restrictive as it only has to be satisfied for a small $\delta > 0$. It allows the estimation of $\gamma$ in a small neighborhood of zero by a quadratic function from above and below. The fourth---together with the negative derivative in zero---ensures that the kernel $\gamma_\eta$ converges to a Dirac distribution. For instance, for $\gamma \equiv \chi_{(0,1)}$, we obtain $\gamma_\eta \equiv \gamma$ for all $\eta > 0$. The fifth is an immediate consequence of the second and third points by choosing a sufficiently small \(\tilde{\delta}\in (0,\delta]\). 
The last point is a normalization as the assumption of \(\gamma\in \sBV(\R_{\geq0})\) implies \(\gamma\in\sBV(\R)\). The choice of  \(\gamma(0)=1\) is arbitrary as \(\gamma_{\eta}\) is (cf. \cref{eqn:cgamma}) anyhow rescaled by the normalization factor \(c_{\eta}\), but allows some notational simplifications in the theorems and proofs to follow.
To illustrate these assumptions, we have visualized them graphically in \cref{fig:ass_gamma}.
\subsection{Local dynamics}\label{subsec:local_dynamics}
As we aim for the singular limit convergence, i.e.,\ whether the solution to the nonlocal dynamics in \cref{defi:nonlocal_conservation_law} converges to its local counterpart if \(\gamma\rightarrow\delta\), we will first need to define what we mean with the local conservation law:
\begin{definition}[Local conservation law]\label{defi:local_conservation_law}
The conservation law related to the corresponding nonlocal conservation law \cref{defi:nonlocal_conservation_law} reads as follows in the variable \(q:\OT\rightarrow\R\)
\begin{equation}
 \begin{aligned}
    q_{t}+\big(V(q(t,x))q(t,x)\big)_{x}&=0,  && (t,x)\in\OT\\
    q(0,x)&=q_{0}(x),  && x\in\R
\end{aligned}
\label{eq:local_conservation_law}
\end{equation}
with \(q_{0}:\R\rightarrow\R\) denoting the initial datum.
\end{definition}

\begin{definition}[Entropy solution]\label{defi:entropy}
We call \(q^{*}\in \sC([0,T];\sL^{1}_{\loc}(\R))\cap \sL^{\infty}((0,T);\sL^{\infty}(\R))\) a weak entropy solution to \cref{eq:local_conservation_law} iff
it satisfies the following entropy inequality
for all \(\alpha\in\sC^{2}(\R)\) being convex, \(\beta\in \sC^{1}(\R)\) with \(\beta'\equiv\alpha'\cdot f'\) and
\[
f(x)=xV(x),\ \forall x\in\R,
\] and for all \(\phi\in \sC_{\text{c}}^{1}\big((-42,T)\times\R;\R_{\geq0}\big)\):
\begin{equation}
\begin{gathered}
\mathcal{E}[\phi,\alpha,q]\coloneqq\iint_{\OT}\!\!\!\!\alpha(q(t,x))\phi_{t}(t,x)+\beta(q(t,x))\phi_{x}(t,x)\dd x\dd t
+\int_{\R}\!\!\alpha(q_{0}(x))\phi(0,x)\dd x\geq 0.
\end{gathered}
\label{eq:Entropy}
\end{equation}
\end{definition}

\begin{theorem}[Existence and uniqueness of entropy solutions]\label{theo:existence_uniqueness_local}
    For the Cauchy problem in \cref{eq:local_conservation_law}, there exists a unique Entropy solution in the sense of \cref{eq:Entropy}.
\end{theorem}
\begin{proof}
    The result can be found for different types of entropy flux pairs in \cite[Theorem 6.3]{bressan}, \cite[Theorem 19.1]{eymard}, \cite[Theorem 2, Theorem 5, Section 5 Item 4]{kruzkov}, \cite{godlewski}.
\end{proof}
\section{Existence/uniqueness nonlocal}\label{sec:existence}
In this section, we will state some relevant well-known results about existence and uniqueness in the presence of weak \(\sBV\) kernels as well as monotonicity violating kernels. This will become crucial for our analysis on the convergence later.

As it is crucial for our analysis, we will, however, give an existence and uniqueness result:
\begin{theorem}[Existence and uniqueness of weak solutions on small time-horizon]\label{lem:existence_uniqueness}
Given \cref{defi:nonlocal_conservation_law} together with \cref{ass:initialdatum_velocity} and \cref{ass:gamma}, it holds: There exists \(T^* \in (0,T]\) 
s.t.\ there is a unique weak solution
\[
q\in \sC\big([0,T^*];\sL^{1}_{\loc}(\R)\big)\cap \sL^{\infty}\big((0,T^*);\sTV(\R)\big).
\]
\end{theorem}
\begin{proof}
The proof can be found in \cite[Theorem 3.1, Theorem 3.3]{keimer2022discontinuous} for an even more general setup as well as in \cite[Theorem 2.20 \& Theorem 3.2 \& Corollary 4.3]{pflug} and, for \(\sBV\) kernels, in \cite{coclite2022existence}.
\end{proof}
A very important result is the stability of the solution when smoothing the initial datum and the velocity to obtain smooth solutions. This is detailed in the following theorem.
\begin{theorem}[Approximation by smooth solutions]\label{theo:stability}
By approximating the initial datum \(q_0\) and the velocity function \(V\) by smooth functions \(q_{0,\eps} \in C^\infty(\R)\) with \(|q_{0,\eps}|_{\sTV(\R)} \leq |q_0|_{\sTV(\R)}\) and \(V_\eps \in C^\infty(\R)\), the corresponding solution \(q_\eps\) for \cref{defi:nonlocal_conservation_law} satisfies \(\sC^\infty(\OT) \cap \sW^{2,\infty}(\OT)\) and approximates the regularized solution \(q\) in \(\sC([0,T];\sL^1(\R))\), i.e., \(\exists c\in \R_{>0}\) s.t. 
\begin{align}
    \|q_\eps - q\|_{\sC([0,T];\sL^1(\R))} \leq c\big(\|q_{0,\eps} - q_0\|_{\sL^1(\R)} + \|V_{\eps} - V\|_{\sW^{1,\infty}(\R)}\big).
\end{align}
\end{theorem}

\begin{proof}
We do not go into details here, and refer the reader to \cite[Theorem 3.1]{pflug4} where such a stability result with regard to the initial datum is proven for kernels with finite support (only a small change is required to extend it to the general kernels considered here) and the method of characteristics. One can additionally smooth the assumed Lipschitz-continuous velocity \(V\) as done in \cite[Theorem 3.2]{keimer2022discontinuous} for a weaker topology, but again relying on characteristics.
\end{proof}

\section{Properties of the nonlocal kernel class considered}\label{sec:kernel}
In order to show the convergence of $q_\eta$ to the corresponding (local) Entropy solution (the singular limit problem), we need to show properties of the considered class of nonlocal kernels. For this, we first study lower and upper bounds for the scaling function $c_\eta$ introduced in \cref{eqn:cgamma}.
\begin{proposition}[Bounds for $c_\eta$ as in \cref{eqn:cgamma}]\label{prop:bounds_cgamma}
There exists $\eps \in \big(0,\frac{1}{2}\big)$ s.t.
\begin{align*}
 -  \tfrac{\gamma'(0)}{\eta} \leq c_\eta \leq   - \tfrac{\gamma'(0)}{\eta(1-2\eta)}, \qquad \forall \eta \in (0,\eps).
\end{align*}

\end{proposition}

\begin{proof}
Due to \cref{ass:gamma}, $\exists \delta \in \R_{>0}$ s.t. $\gamma|_{(0,\delta)} \in \sW^{2,\infty}((0,\delta))$ and, thus, we can bound $\gamma$ on $(0,\delta)$ from above and below by a linear function because the second derivative of $\gamma$ is bounded on $(0,\delta)$, i.e., $\alpha \: \|\gamma''\|_{\sL^\infty((0,\delta))}$. For all $\hat \delta \in (0,\min\{\delta,\gamma']$ we obtain, by a Taylor expansion,
\begin{align}
1+x\big(\gamma'(0) - \tfrac{\alpha }{2} \hat \delta\big)&\leq \gamma(x) \leq 1+x\big(\gamma'(0) + \tfrac{\alpha }{2} \hat \delta\big)&& \forall x\in (0,\hat \delta). \label{eq:linear_bounds_gamma}
\end{align}
These estimates are visualized in \cref{fig:gamma_est}.
\begin{figure}
\centering
\begin{tikzpicture}
\begin{axis}[height=5cm,width=10cm,xmin=0,ymin=-0.7,xmax=1,ymax=1,grid=both,grid style=dotted,axis line style={line width=0.25pt},xtick={0,0.6,2},
xticklabels={$0$,$\hat \delta$,$\delta$},xlabel={$x$},ytick={0.1556,1},yticklabels={$\gamma(\delta)$,$\gamma(0)=1$},legend style={draw=none,at={(0.02,0.02)},anchor=south west,font=\plotfont},legend cell align={left},tick label style={font=\plotfont}]
\addplot[col2,domain=0:2,samples=100,line width=1pt]  {(0.6+0.4/(1+x/4)^3)*3-2};
\addplot[col1,domain=0:2,samples=100,line width=.5pt,style=solid]  {(1-1.2*(x/4)+4.8*(x/4)^2)*3-2};
\addplot[col3,line width=1pt]coordinates  {(0,1)(0.6,0.7840)};
\addplot[col3,line width=.5pt,dashed]coordinates  {(0,1)(10/9,0)};
\addplot[col1,domain=0:2,samples=100,line width=.5pt,style=solid]  {(1-1.2*(x/4)-4.8*(x/4)^2)*3-2};
\addplot[col3,line width=1pt]coordinates  {(0,1)(0.6,0.1360)};
\addplot[only marks,col3,line width=1pt,mark=o]coordinates  {(0.6,0.7840)(0.6,0.1360)};
\legend{$\gamma$, $x \mapsto 1+\gamma'(0)x\pm\frac{\alpha}{2}x^2$, $x \mapsto 1+(\gamma'(0)\pm\frac{\alpha}{2}\hat \delta)x$,$x\mapsto 1+\gamma'(0)x$}
\end{axis}
\end{tikzpicture}
\caption{Sketch to visualize the ideas of the first steps in the proof of \cref{prop:bounds_cgamma} showing the \textcolor{col1}{quadratic} and \textcolor{col3}{linear} estimators for the kernel $\gamma$. }\label{fig:gamma_est}
\end{figure}
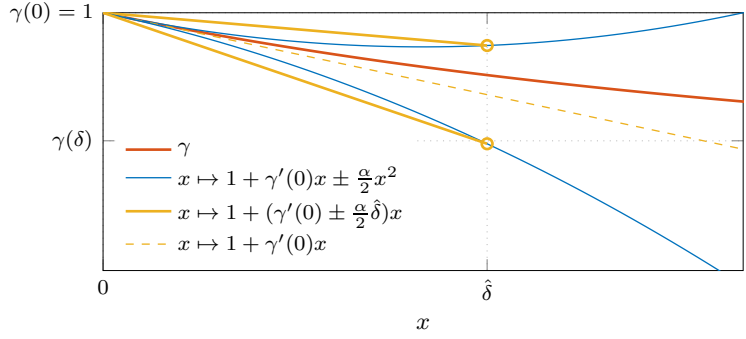
 First, we consider the lower bound in \cref{eq:linear_bounds_gamma} and w.l.o.g.\ assume $\hat \delta$ small enough s.t.\ the left hand side of \cref{eq:linear_bounds_gamma} is non-negative and obtain 
\begin{align}
       \tfrac{1}{\eta c_\eta} &= \tfrac{1}{\eta}\int_{0}^{\infty} \gamma(x)^{\frac{1}{\eta}} \dd x  \geq  \tfrac{1}{\eta}\int_{0}^{\hat \delta} \Big(1+x\big(\gamma'(0) - \tfrac{\alpha }{2} \hat \delta\big)\Big)^{\frac{1}{\eta}} \dd x  \\
       &= \tfrac{\big(1+\hat \delta \big(\gamma'(0) - \tfrac{\alpha }{2} \hat \delta\big)\big)^{\frac{1}{\eta}+1}-1}{\big(\gamma'(0) - \tfrac{\alpha }{2} \hat \delta\big)(1+\eta)} \label{eq:lb_01}\\
    \intertext{for $\eta$ small enough, we can chose $\hat \delta = \eta^{\frac{3}{4}}$}
    &   = \tfrac{\big(1+\eta^{\frac{3}{4}}\gamma'(0) - \eta^{\frac{3}{2}}\tfrac{\alpha }{2} \big)^{\frac{1}{\eta}+1}-1}{\big(\gamma'(0) - \tfrac{\alpha }{2} \eta^{\frac{3}{4}}\big)(1+\eta)} \quad \stackrel{\eta \rightarrow 0}{\longrightarrow} \quad -\tfrac{1}{\gamma'(0)}. \label{eq:lb_02}
\end{align}
Analogously, for the upper bound we obtain 
\begin{align}
       \tfrac{1}{\eta c_\eta} &= \tfrac{1}{\eta}\int_{0}^{\infty} \gamma(x)^{\frac{1}{\eta}} \dd x  \\
       &\leq \tfrac{1}{\eta}\int_{0}^{\hat \delta} \Big(1+x\big(\gamma'(0) + \tfrac{\alpha }{2} \hat \delta\big)\Big)^{\frac{1}{\eta}} \dd x  + \tfrac{1}{\eta} \gamma(\hat \delta)^{\frac{1}{\eta}-1}\int_{\hat \delta}^{\infty} \gamma(x) \dd x \label{eq:ub_01}\\
       &\leq \tfrac{\big(1+\hat \delta \big(\gamma'(0) + \tfrac{\alpha }{2} \hat \delta\big)\big)^{\frac{1}{\eta}+1}-1}{\big(\gamma'(0) + \tfrac{\alpha }{2} \hat \delta\big)(1+\eta)} + \Big(1+\hat \delta\big(\gamma'(0) + \tfrac{\alpha }{2} \hat \delta\big)\Big)^{\frac{1}{\eta}-1} \|\gamma\|_{\sL^1(\R_{>0})}\notag\\
    \intertext{again, choosing for $\eta$ small enough $\hat \delta = \eta^{\frac{3}{4}}$}
    &   = \tfrac{\big(1+\eta^{\frac{3}{4}}\gamma'(0) + \eta^{\frac{3}{2}}\tfrac{\alpha }{2} \big)^{\frac{1}{\eta}+1}-1}{\big(\gamma'(0) + \tfrac{\alpha }{2} \eta^{\frac{3}{4}}\big)(1+\eta)} + \big(1+\eta^{\frac{3}{4}}\gamma'(0) + \eta^{\frac{3}{2}}\tfrac{\alpha }{2} \big)^{\frac{1}{\eta}-1} \|\gamma\|_{\sL^1(\R_{>0})}  \\
    &\stackrel{\eta \rightarrow 0}{\longrightarrow} \quad -\tfrac{1}{\gamma'(0)}.\label{eq:ub_03}
\end{align}
As both limits coincide, we thus obtain
\begin{align}
  \lim_{\eta \rightarrow 0}  \tfrac{1}{\eta c_\gamma(\eta)} = -\tfrac{1}{\gamma'(0)} \label{eqn:limit_cgamma}
\end{align}
To show that this convergence is of order $\mathcal O(\eta)$ for $\eta \rightarrow 0$, we recall the inequalities \eqref{eq:lb_01} to \eqref{eq:lb_02} and \eqref{eq:ub_01} to \eqref{eq:ub_03} for small enough \(\eta\), adding $\gamma'(0)^{-1}$ to obtain
\begin{align*}
\tfrac{\big(1+\eta^{\frac{3}{4}}\gamma'(0) - \eta^{\frac{3}{2}}\tfrac{\alpha }{2} \big)^{\frac{1}{\eta}+1}-1}{\big(\gamma'(0) - \tfrac{\alpha }{2} \eta^{\frac{3}{4}}\big)(1+\eta)} + \tfrac{1}{\gamma'(0)} \leq\tfrac{1}{\eta c_{\gamma}(\eta)}+\tfrac{1}{\gamma'(0)} \leq \tfrac{\big(1+\eta^{\frac{3}{4}}\gamma'(0) + \eta^{\frac{3}{2}}\tfrac{\alpha }{2} \big)^{\frac{1}{\eta}+1}-1}{\big(\gamma'(0) - \tfrac{\alpha }{2} \eta^{\frac{3}{4}}\big)(1+\eta)} + \tfrac{1}{\gamma'(0)}.
\end{align*}
First, we have a look at the lower bound in the previous estimate, divide by $\eta$, and conclude
\begin{align}
\lim_{\eta \rightarrow 0}\tfrac{1}{\eta}\Bigg(\tfrac{\big(1+\eta^{\frac{3}{4}}\gamma'(0) - \eta^{\frac{3}{2}}\tfrac{\alpha }{2} \big)^{\frac{1}{\eta}+1}-1}{\big(\gamma'(0) - \tfrac{\alpha }{2} \eta^{\frac{3}{4}}\big)(1+\eta)} + \tfrac{1}{\gamma'(0)}\Bigg) 
% &= \lim_{\eta \rightarrow 0}\tfrac{1}{\eta}\Bigg(\frac{\big(1+\eta^{\frac{3}{4}}\gamma'(0)  \big)^{\frac{1}{\eta}+1}-1}{\gamma'(0)(1+\eta)} + \frac{1}{\gamma'(0)}\Bigg) \\
 &= \lim_{\eta \rightarrow 0}\tfrac{\big(1+\eta^{\frac{3}{4}}\gamma'(0)  \big)^{\frac{1}{\eta}+1}+ \eta}{\gamma'(0)(1+\eta)\eta} 
  = \tfrac{1}{\gamma'(0)}.
\end{align}
 The analogous estimate with the same limit can be performed for the upper bound. Thus, for each $\eps \in \R_{>0}$, there exists $\bar \eta \in \R_{>0}$ s.t.
 \begin{align*}
\tfrac{1}{\gamma'(0)}-\eps\leq\tfrac{1}{\eta} \big(\tfrac{1}{\eta c_{\gamma}(\eta)}+\tfrac{1}{\gamma'(0)}\big) \leq \tfrac{1}{\gamma'(0)}+\eps \qquad \forall \eta \in (0,\bar \eta)
\end{align*}
 Consequently, there exists $\tilde \eta \in \R_{>0}$ small enough s.t. 
\begin{align*}
\tfrac{2}{\gamma'(0)}\leq\tfrac{1}{\eta} \big(\tfrac{1}{\eta c_\eta}+\tfrac{1}{\gamma'(0)}\big) \leq 0 \qquad \forall \eta \in (0,\tilde \eta)
\end{align*}
and, consequently (keeping in mind that $\gamma'(0)<0$ by \cref{ass:gamma}),
\begin{align*}
\tfrac{2 \eta}{\gamma'(0)} &\leq \tfrac{1}{\eta c_\eta}+\tfrac{1}{\gamma'(0)} \leq 0 \qquad
%\Leftrightarrow  \qquad \tfrac{2 \eta - 1}{\gamma'(0)}  &\leq \tfrac{1}{\eta c_{\gamma}(\eta)} \leq -\tfrac{1}{\gamma'(0)}\\
%Leftrightarrow  \qquad \tfrac{\eta(2 \eta - 1)}{\gamma'(0)}  &\leq \tfrac{1}{c_{\gamma}(\eta)} \leq -\tfrac{\eta}{\gamma'(0)}
\Longleftrightarrow  \qquad -\tfrac{\gamma'(0)}{\eta(1-2 \eta)}  \geq c_\eta \geq -\tfrac{\gamma'(0)}{\eta} \qquad \forall \eta \in (0,\tilde \eta)
\end{align*}
which concludes the proof.
\end{proof}

\begin{proposition}[Property of the nonlocal weight and its derivative] \label{prop:prop_gamma_deriv}
Let $\gamma$ satisfy \cref{ass:gamma} and $a \in \sL^\infty((0,1);\R_{\geq 0})$ be given. Then, there exist $\mD_1,\mD_2 \in \R_{>0}$ and $\hat \eta \in (0,1]$ s.t.\ the following inequality holds:
\begin{align}
  \Big\|\gamma_\eta + \eta a(\eta)\gamma_\eta'\Big\|_{\sL^1(\R_{>0})} &\leq \eta \mD_1+\big|1+a(\eta)\gamma'(0)\big| \mD_2\qquad\forall \eta \in (0,\hat \eta) \label{eq:the_key_equation}.
\end{align}
\end{proposition}

\begin{proof}
We estimate as follows:
\begin{align}    &\tfrac{1}{c_\eta}\int_{0}^\infty \big|\gamma_\eta(x) + \eta a(\eta)\gamma_{\eta}'(x)\big| \dd x=  \int_{0}^\infty \Big|\gamma(x)^{\frac{1}{\eta}} + a(\eta)\gamma(x)^{\frac{1}{\eta}-1}\gamma'(x)\Big| \dd x \notag\\
    %&=  c_\eta\int_{0}^\infty \gamma(x)^{\frac{1}{\eta}-1}\Big|\gamma(x) + \tfrac{\eta}{\gamma'(0)} \gamma'(x)\Big| \dd x\\
    \intertext{using \cref{ass:gamma}, Item 2 to split the integral from \((0,\hat{\delta})\) with \(\hat{\delta}\in(0,\delta]\) small enough}
&\leq  \int_{0}^{\hat \delta} \gamma(x)^{\frac{1}{\eta}-1}\big|\gamma(x) + a(\eta) \gamma'(x)\big| \dd x  + \int_{\hat \delta}^\infty \gamma(x)^{\frac{1}{\eta}-1}\big|\gamma(x) + a(\eta) \gamma'(x)\big| \dd x\notag\\
\intertext{using the fundamental theorem of differentiation, integration on the mapping \(x \mapsto \gamma(x)+\alpha(\eta)\gamma'(x)\) for the first summand, and \cref{ass:gamma} Item 4 for the second integral}
 &\leq \int_{0}^{\hat \delta}\gamma(x)^{\frac{1}{\eta}-1}\bigg|\gamma(0)+a(\eta)\gamma'(0)  + \int_0^x \gamma'(s) + a(\eta)\gamma''(s) \dd s \bigg| \dd x\notag\\
 &\quad+  \big\|\gamma + a(\eta)\gamma'\big\|_{\sL^1((0,\infty))} \gamma(\hat \delta)^{\frac{1-\eta}{\eta}}\notag
 \intertext{using the triangular inequality and in some cases the estimate \(\gamma\leqq \gamma(0)=1\) (\cref{ass:gamma}, Item 4}
 &\leq \underbrace{|\gamma(0)+a(\eta)\gamma'(0)|\vphantom{\bigg|}\int_{0}^{\hat \delta}\!\!\!\gamma(x)^{\frac{1-\eta}{\eta}} \dd x}_{\coloneqq A}+ \underbrace{\vphantom{\bigg|}\|\gamma' + a(\eta)\gamma''\|_{\sL^\infty((0,\hat \delta)}\int_{0}^{\hat \delta} \!\!\! x\gamma(x)^{\frac{1-\eta}{\eta}} \dd x}_{\coloneqq B} \notag\\
 &\quad+  \underbrace{\vphantom{\bigg|}\big\|\gamma + a(\eta)\gamma'\big\|_{\sL^1(\R_{>0})}  \gamma(\hat \delta)^{\frac{1-\eta}{\eta}}}_{\coloneqq C}\label{eq:defi_terms}
\end{align}
Using \cref{eq:linear_bounds_gamma}, we obtain particularly for $\hat \delta \coloneqq \min\big\{-\gamma'(0) \|\gamma''\|_{\sL^\infty((0,\delta))}^{-1}\ ,\ \delta\big\}$ and $\kappa \coloneqq \frac{1}{2}\gamma'(0) < 0$ as upper bound on \(\gamma\) locally around zero
\begin{align*}
 \gamma(x) \leq  1+\kappa x,&& \forall x\in (0,\hat \delta),
\end{align*}
a linear upper bound with negative slope.
We estimate three terms defined in \cref{eq:defi_terms} as follows: 
\begin{align*}
    A &\leq \big|\gamma(0)+a(\eta)\gamma'(0)\big|\int_{0}^{\hat \delta}\big(1+ \kappa x\big)^{\frac{1}{\eta}-1} \dd x 
=  \eta \big|1+a(\eta)\gamma'(0)\big| \tfrac{(1+ \kappa \hat \delta )^{\frac{1}{\eta}}-1}{\kappa} \\
   B& \leq \|\gamma' + a(\eta) \gamma''\|_{\sL^\infty((0,\hat \delta))} \int_0^{\hat \delta} (1+\kappa x)^{\frac{1}{\eta}-1}x \dd x \\
   &=\eta^2  \|\gamma' + a(\eta) \gamma''\|_{\sL^\infty((0,\hat \delta))}\tfrac{1-(1 + \kappa \hat \delta )^{\frac{1}{\eta}} (1 -   \hat \delta \frac{\kappa}{\eta} )}{\kappa^2  (1+\eta)} \\
   C & \leq   \big\|\gamma + a(\eta)\gamma'\big\|_{\sL^1((0,\infty))} (1+\kappa \hat \delta)^{\frac{1-\eta}{\eta}} \leq  \big(\|\gamma\|_{\sL^1(\R_{>0})}+a(\eta) |\gamma|_{\sTV(\R_{>0})}\big)(1+\kappa \hat \delta)^{\frac{1-\eta}{\eta}}.
\end{align*}
Combining the previously obtained estimates yields
\begin{align*}
\Big\|\gamma_\eta + \eta a(\eta)\gamma_\eta'\Big\|_{\sL^1(\R_{>0})} 
&\leq  c_\eta (A+B+C) \\
&= c_\eta\eta \big|1+a(\eta)\gamma'(0)\big| \tfrac{(1+ \kappa \hat \delta )^{\frac{1}{\eta}}-1}{\kappa}  \\
    & \quad + c_\eta\eta^2 \bigg( \|\gamma' + a(\eta) \gamma''\|_{\sL^\infty((0,\hat \delta))}\tfrac{1-(1 + \kappa \hat \delta )^{\frac{1}{\eta}} (1 -   \hat \delta \frac{\kappa}{\eta} )}{\kappa^2  (1+\eta)} \\
    &\qquad\qquad \qquad 
    + \tfrac{1}{\eta^2}\big(\|\gamma\|_{\sL^1(\R_{>0})}+a(\eta) |\gamma|_{\sTV(\R_{>0})}\big)(1+\kappa \hat \delta)^{\frac{1-\eta}{\eta}} \bigg)
\end{align*}
together with the bounds on $c_\eta$, i.e., $c_\eta \in \mathcal O(\eta^{-1})$, in \cref{prop:bounds_cgamma} and making use of $1+\kappa \hat \delta< 1$, the claimed estimate holds with
\begin{align*}
    \mathcal D_1 &\coloneqq -2\gamma'(0)\Big( \big(\|\gamma'\|_{\sL^\infty((0,\hat \delta))} + \|a\|_{\sL^\infty((0,1))} \|\gamma''\|_{\sL^\infty((0,\hat \delta))}\big)\tfrac{1}{\kappa^2} + 1\Big)  \quad \text{and} \quad \mathcal D_2 \coloneqq 2\tfrac{\gamma'(0)}{\kappa}.
\end{align*}
\end{proof}

 \begin{remark}[The exponential kernel]
In the special case of the exponential kernels, i.e.\ $\gamma \equiv \exp(-\cdot)$ we obtain the following identity: 
\begin{align}
    \gamma_\eta \equiv -\eta \gamma_\eta'
\end{align}
and thus the estimate in \cref{prop:prop_gamma_deriv} is satisfied with $\mathcal D_1 = 0$ and $\mathcal D_2 = 0$. This is in line with the singular limit for the exponential case \cite{coclite2022general}. 
 \end{remark}

\section{Weakened maximum principle}

The assumptions on the kernel do not imply monotonicity on its support and, consequently, the maximum principles in recent literature \cite{blandin2016well,scialanga,pflug,coclite2022existence} are not applicable to extending the solution onto the full time horizon \([0,T]\). However, we can show the following weakened maximum principle which will also serve us to extend the weak solution from a small time horizon to an arbitrary large time horizon. For an illustration, we refer the reader to \cref{fig:max_q}.

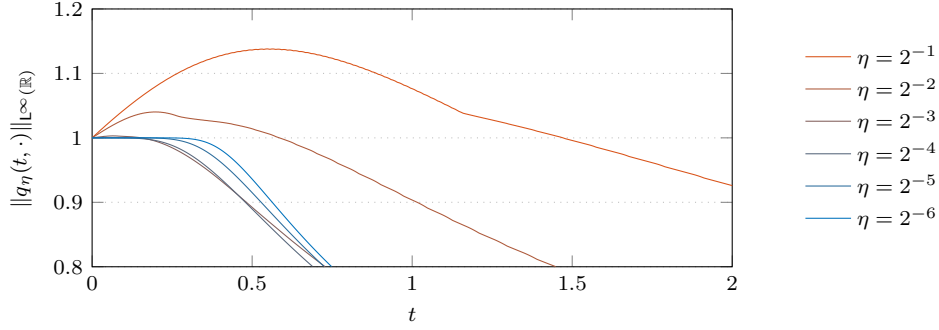
\begin{figure}
    \centering
    \pgfplotstableread{data/maximum_principle1.txt}{\maxprinciple}
\begin{tikzpicture}
\begin{axis}[ymajorgrids,
grid style=dotted,height=5cm,width=10cm,
          xmin=0,ymin=0.8,
          xmax=2,ymax=1.2,
          restrict y to domain=0.8:1.2,
          restrict x to domain=0:2,
          xtick={0.0,0.5,1.0,1.5,2.0},
          ytick={0.8,0.9,1.0,1.1,1.2},xlabel={$t$},ylabel={\(\|q_\eta(t,\cdot)\|_{\sL^\infty(\R)}\)},
          legend style={draw=none,fill=none,font=\plotfont,at={(1.1,0.9)},anchor=north west},legend cell align={left},tick label style={font=\plotfont}, label style={font=\plotfont}]

\addplot[col1!0!col2] table [x = {t}, y = {eta1}] {\maxprinciple};
\addplot[col1!20!col2] table [x = {t}, y = {eta2}] {\maxprinciple};
\addplot[col1!40!col2] table [x = {t}, y = {eta3}] {\maxprinciple};
\addplot[col1!60!col2] table [x = {t}, y = {eta4}] {\maxprinciple};
\addplot[col1!80!col2] table [x = {t}, y = {eta5}] {\maxprinciple};
\addplot[col1!100!col2] table [x = {t}, y = {eta6}] {\maxprinciple};

\legend{$\eta = 2^{-1}$,$\eta = 2^{-2}$,$\eta = 2^{-3}$,$\eta = 2^{-4}$,$\eta = 2^{-5}$,$\eta = 2^{-6}$}
\end{axis}
\end{tikzpicture}
    \caption{Visualization of the weakened maximum principle as proven in \cref{theo:max}, showing an increase of \(t \mapsto \|q_\eta(t,\cdot)\|_{\sL^\infty(\R)}\) over time $t$ which vanishes for \(\eta \rightarrow 0\). The initial datum \(q_0 \equiv \chi_{(-1,-0.5)} + \chi_{(1,1.5)}\), \(\gamma \equiv (1-2 \cdot) \chi_{(0,0.5)} + \tfrac{1}{2} \chi_{(0.5,2)}\) and \(V \equiv 1 - \cdot\) are chosen s.t.\ the ``classical'' maximum principle is violated.}
    \label{fig:max_q}
\end{figure}

\begin{theorem}[Weakened maximum principle] \label{theo:max}
For $\eta \in \R_{>0}$ let $q_\eta$ be the unique weak solution to \cref{defi:nonlocal_conservation_law} as stated in \cref{lem:existence_uniqueness} until a given time \(T\in\R_{>0}\). Then, for \(\kappa\in \R_{>0}\), there exists \( \eta_{\kappa,T} \in \R_{>0}\) s.t.
\[ 
\forall (t,\eta) \in (0,T)\times \big(0,\eta_{\kappa,T}\big)\ : \quad \|q_\eta(t,\cdot)\|_{\sL^\infty(\R)} \leq (1+\kappa) \|q_0\|_{\sL^\infty(\R)}.
\]
\end{theorem}
\begin{proof}
We prove the maximum principle with a sequence of smooth solutions which, according to \cref{theo:stability}, exist on a small enough time horizon. As the bounds we derive on the smoothed solutions are uniform with regard to the smoothing parameter, the bounds hold also for the non-smoothed setting. In the following manipulations, we assume directly smooth solutions, i.e.,\ \(q\in \sC^{\infty}(\OT)\) with \(q\in W^{2,\infty}(\OT)\). 

To prove this, let \(\epsilon \in \R_{\geq 0}\) be given and define \[f(t+\eps) \coloneqq \sup_{y\in \R}q(t+\eps,y)\text{ and } X(\eps) \coloneqq \{x\in \R \cup \{-\infty,\infty\} \ : q(t+\eps,x) = f(\eps)\}\] . Then we obtain for every \(x_\epsilon \in X(\epsilon)\) and \(x\in X(0)\)
\begin{align}
    f(t+\eps) - f(t)
    &= q(t+\epsilon,x_\epsilon)-q(t,x_\epsilon)+\underbrace{q(t,x_\epsilon)-q(t,x)}_{\leq 0}\\
    &\leq q(t+\epsilon,x_\epsilon)-q(t,x_\epsilon) \leq \epsilon \sup_{y\in X(\epsilon)} \partial_t q(t,y) + \epsilon^2 \|\partial_1^2 q \|_{\sL^\infty(\OT)}.
\end{align}
where the letter inequality for the case \(|x_\eps| = \infty\) has to be interpreted as a limit. As \(\partial_1 q_0  \in \sL^{1}(\R) \) we obtain \(\partial_2 q \in \sW^{1,1}(\R)\).
By \(\partial_1 q = -V'(W)\partial_2 W q - V(W)\partial_2 q\) we thus have that \(\lim_{y\rightarrow \infty} \partial_t q(t,y)\) is well defined.

As, by construction, \(f \in W^{1,\infty}((0,T))\) we can 
divide the letter inequality by \(\epsilon\) and, passing to the limit, i.e., letting \(\epsilon \searrow 0\), then yields, for \(t\in[0,T]\),
\begin{align}\label{eq:change_sup_derivative}
    \partial_t f(t) \leq \limsup_{\epsilon \searrow 0} \sup_{y \in X(\epsilon)} \partial_t q(t,y)
\end{align}
where the derivative has to be interpreted in a weak sense. As the second spatial derivative of $q$ is bounded, we can estimate $q$ as follows for all \(y\in \R\):
 \begin{align}
     q(t,x) &\geq q(t,x_\epsilon) + \partial_2 q(t,x_\epsilon) (y-x_\epsilon) - \tfrac{1}{2}\|\partial_2^2 q\|_{\sL^\infty(\OT)}(y-x_\epsilon)^2 \\
     \intertext{as $q(t,x_\epsilon) \geq q(t+\epsilon,x_\epsilon) - \epsilon \|\partial_t q\|_{\sL^\infty(\OT)} \geq \big(q(t,x) - \epsilon \|\partial_t q\|_{\sL^\infty(\OT)}\big) - \epsilon \|\partial_t q\|_{\sL^\infty(\OT)}$ we obtain}
     &\geq q(t,x) - 2\epsilon \|\partial_t q\|_{\sL^\infty(\OT)} + \partial_2 q(t,x_\epsilon) (y-x_\epsilon) - \tfrac{1}{2}\|\partial_2^2 q\|_{\sL^\infty(\OT)}(y-x_\epsilon)^2.
 \end{align}
Reordering the terms, we obtain
 \begin{align}
     0 &\leq  2\epsilon \|\partial_t q\|_{\sL^\infty(\OT)} - \partial_2 q(t,x_\epsilon) (y-x_\epsilon) + \tfrac{1}{2}\|\partial_2^2 q\|_{\sL^\infty(\OT)}(y-x_\epsilon)^2 . 
 \end{align}
 As these estimates are uniform in $y$, we obtain---by making the right hand side of the latter equation minimal---the following:
  \begin{align*}
     0 &\leq  2\epsilon \|\partial_1 q\|_{\sL^\infty(\OT)}  - \tfrac{\partial_2 q(t,x_\epsilon)^2}{\|\partial_2^2 q\|_{\sL^\infty(\OT)}} + \tfrac{1}{2}\|\partial_2^2 q\|_{\sL^\infty(\OT)}\tfrac{\partial_2 q(t,x_\epsilon)^2}{\|\partial_2^2 q\|_{\sL^\infty(\OT)}^2} \\
     &=2\epsilon \|\partial_t q\|_{\sL^\infty(\OT)} - \tfrac{\partial_2 q(t,x_\epsilon)^2}{2\|\partial_2^2 q\|_{\sL^\infty(\OT)}\!}
 \end{align*}
 and consequently
     \begin{equation}\partial_2 q(t,x_\epsilon)^2\leq 4\epsilon \|\partial_t q\|_{\sL^\infty(\OT)}\|\partial_2^2 q\|_{\sL^\infty(\OT)}.\label{eq:d_2_q_eps_2}\end{equation}
Based on these preliminary calculations, we can now estimate $\partial_t q(t,x_\eps)$ for $x_\epsilon\in X(\eps)$. For the sake of simplicity, we write \(W \equiv W[q,\gamma_\eta]\):
    \begin{align}
        q_t(t,x_\epsilon ) &= -V'(W(t,x_\epsilon ))\partial_x W(t,x_\epsilon )q(t,x_\epsilon ) -V(W(t,x_\epsilon ))\partial_{x}q(t,x_\epsilon )\notag
        \intertext{using the definition of  \(W\), i.e.\ \cref{eqn:nonlocalterm}, and the previously derived estimate in \cref{eq:d_2_q_eps_2}}
        &\leq V'(W(t,x_\epsilon ))c_\eta\bigg(q(t,x_\epsilon ) + \tfrac{1}{\eta}\int_{x}^{\infty} \!\!\!\gamma(y-x)^\frac{1-\eta}{\eta} \gamma'(y-x) q(t,y)\dd y \bigg) q(t,x_\epsilon ) \notag\\
        &\qquad + 2V(W(t,x_\epsilon ))\Big(\epsilon \|\partial_t q\|_{\sL^\infty(\OT)}\|\partial_2^2 q\|_{\sL^\infty(\OT)}\Big)^\frac{1}{2}\notag\\
        \intertext{and splitting the integral into two terms with $\delta \in \R_{>0}$ as in \cref{ass:gamma}}
         &= V'(W(t,x_\epsilon ))c_\eta\bigg(q(t,x_\epsilon ) + \tfrac{1}{\eta}\int_{x}^{x+\delta} \gamma(y-x)^\frac{1-\eta}{\eta} \gamma'(y-x) q(t,y)\dd y \notag \\
         &\qquad \qquad\qquad \qquad \qquad + \tfrac{1}{\eta}\int_{x+\delta}^{\infty} \gamma(y-x)^\frac{1-\eta}{\eta} \gamma'(y-x) q(t,y)\dd y \bigg) q(t,x_\epsilon )  \notag\\
        &\qquad + 2V(W(t,x_\epsilon )\Big(\epsilon \|\partial_t q\|_{\sL^\infty(\OT)}\|\partial_2^2 q\|_{\sL^\infty(\OT)}\Big)^\frac{1}{2}\notag\\         
         \intertext{as \(\gamma'|_{(0,\delta)} \leq  0 \) by \cref{ass:gamma}  and \(q(t,x_\epsilon )\leq f(t)\) we obtain}
         &\leq V'(W(t,x_\epsilon ))c_\eta\bigg(q(t,x_\epsilon )  + f(t)\tfrac{1}{\eta}\int_{x}^{x+\delta} \gamma(y-x)^\frac{1-\eta}{\eta} \gamma'(y-x) \dd y\notag  \\   
         &\qquad \qquad  \qquad \qquad \qquad - \tfrac{1}{\eta}\gamma(\delta)^\frac{1-\eta}{\eta} f(t)\int_{x+\delta}^{\infty} |\gamma'(y-x)| \dd y \bigg) q(t,x_\epsilon )  \notag\\    
         &\qquad + 2V(W(t,x_\epsilon ))\Big(\epsilon \|\partial_t q\|_{\sL^\infty(\OT)}\|\partial_2^2 q\|_{\sL^\infty(\OT)}\Big)^\frac{1}{2}\notag\\
           &\leq V'(W(t,x_\epsilon ))c_\eta q(t,x_\epsilon )^2  \notag\\
           &\qquad +  V'(W(t,x_\epsilon ))c_\eta\Big(\gamma(\delta)^\frac{1}{\eta} - \gamma(0)^\frac{1}{\eta} - \tfrac{1}{\eta}\gamma(\delta)^\frac{1-\eta}{\eta}  |\gamma|_{\sTV(\R_{>0})} \Big)f(t)q(t,x_\epsilon)\notag\\
           &\qquad + 2V(W(t,x_\epsilon ))\Big(\epsilon \|\partial_t q\|_{\sL^\infty(\OT)}\|\partial_2^2 q\|_{\sL^\infty(\OT)}\Big)^\frac{1}{2}\notag
           \intertext{as \(f(t)\geq q(t,x_\epsilon) \geq f(t)-2\varepsilon \|\partial_1 q\|_{\sL^\infty(\OT)}\) we can make the r.h.s.\ uniform in \(x_\eps\)}
           &\leq V'(W(t,x_\epsilon ))c_\eta \big(f(t)-2\varepsilon \|\partial_1 q\|_{\sL^\infty(\OT)}\big)^2 \notag \\
           &\qquad +  V'(W(t,x_\epsilon ))c_\eta\Big(\gamma(\delta)^\frac{1}{\eta} - 1 - \tfrac{1}{\eta}\gamma(\delta)^\frac{1-\eta}{\eta}  |\gamma|_{\sTV(\R_{>0})} \Big)f(t)^2\notag\\
           &\qquad + 2V(W(t,x_\epsilon ))\Big(\epsilon \|\partial_t q\|_{\sL^\infty(\OT)}\|\partial_2^2 q\|_{\sL^\infty(\OT)}\Big)^\frac{1}{2}   \notag  \\
           &= V'(W(t,x_\epsilon ))c_\eta \big(-4f(t)\varepsilon \|\partial_1 q\|_{\sL^\infty(\OT)}+4\varepsilon^2 \|\partial_1 q\|_{\sL^\infty(\OT)}^2\big)  \notag\\
           &\qquad +  V'(W(t,x_\epsilon ))c_\eta\Big(\gamma(\delta)^\frac{1}{\eta} - \tfrac{1}{\eta}\gamma(\delta)^\frac{1-\eta}{\eta}  |\gamma|_{\sTV(\R_{>0})} \Big)f(t)^2\notag\\
           &\qquad + 2V(W(t,x_\epsilon ))\Big(\epsilon \|\partial_t q\|_{\sL^\infty(\OT)}\|\partial_2^2 q\|_{\sL^\infty(\OT)}\Big)^\frac{1}{2} \notag
           \intertext{ leaving out the first term in the second line as it is negative---and making the estimates uniform in \(\eps\)}
            &\leq \|V'\|_{\sL^\infty((0,2\|q_0\|_{\sL^\infty(\R))}))} c_\eta \big|-4f(t)\varepsilon \|\partial_1 q\|_{\sL^\infty(\OT)}+4\varepsilon^2 \|\partial_1 q\|_{\sL^\infty(\OT)}^2\big| \label{eq:q_t_estimate_1} \\
           &\qquad + \|V'\|_{\sL^\infty((0,2\|q_0\|_{\sL^\infty(\R))}))}c_\eta\tfrac{1}{\eta}\gamma(\delta)^\frac{1-\eta}{\eta}  |\gamma|_{\sTV(\R_{>0})}f(t)^2\\
           &\qquad + 2\|V\|_{\sL^\infty((0,2\|q_0\|_{\sL^\infty(\R))}))}\Big(\epsilon \|\partial_t q\|_{\sL^\infty(\OT)}\|\partial_2^2 q\|_{\sL^\infty(\OT)}\Big)^\frac{1}{2}\label{eq:q_t_estimate_3}. 
    \end{align}
Plugging the latter estimate for $\partial_t q(t,x_\epsilon)$ into \cref{eq:change_sup_derivative}, we obtain
\begin{align*}
    \partial_t f(t) &\leq \limsup_{\eps\searrow 0}  \bigg(\|V'\|_{\sL^\infty((0,2\|q_0\|_{\sL^\infty(\R))}))} c_\eta \big|-4f(t)\varepsilon \|\partial_1 q\|_{\sL^\infty(\OT)}+4\varepsilon^2 \|\partial_1 q\|_{\sL^\infty(\OT)}^2\big|  \\
           &\qquad\qquad\qquad + \|V'\|_{\sL^\infty((0,2\|q_0\|_{\sL^\infty(\R))}))}c_\eta\tfrac{1}{\eta}\gamma(\delta)^\frac{1-\eta}{\eta}  |\gamma|_{\sTV(\R_{>0})}f(t)^2\\
           &\qquad\qquad\qquad + 2\|V\|_{\sL^\infty((0,2\|q_0\|_{\sL^\infty(\R))}))}\Big(\epsilon \|\partial_t q\|_{\sL^\infty(\OT)}\|\partial_2^2 q\|_{\sL^\infty(\OT)}\Big)^\frac{1}{2} \bigg)\\ 
     &=\|V'\|_{\sL^\infty((0,2\|q_0\|_{\sL^\infty(\R))}))}c_\eta\tfrac{1}{\eta}\gamma(\delta)^\frac{1-\eta}{\eta}  |\gamma|_{\sTV(\R_{>0})}f(t)^2\ \forall t\in[0,T].
\end{align*}
    Thus, we obtain, by the comparison principle,
    \begin{align*}
    f(t) &\leq \tfrac{\|q_{0}\|_{\sL^{\infty}(\R)}}{1-t C_\eta \|q_{0}\|_{\sL^{\infty}(\R)}}, \qquad\qquad \forall t \in \left[0,C_\eta^{-1} \|q_{0}\|_{\sL^{\infty}(\R)}^{-1}\right)
    \intertext{with}
    C_\eta &\: \|V'\|_{\sL^\infty((0,(1+\eps)\|q_0\|_{\sL^\infty(\R)}))}  \tfrac{c_\eta}{\eta} \gamma(\delta)^{\frac{1-\eta}{\eta}} |\gamma|_{\sTV(\R)}.
\end{align*}
Consequently, we obtain, for every \(\kappa \in \R_{>0}\),
\begin{equation}
\forall t\in \Big[0,\tfrac{\kappa}{(1+\kappa)}C_\eta^{-1}\|q_{0}\|^{-1}_{\sL^{\infty}(\R)}\Big]:\ 
\|q(t,\cdot)\|_{\sL^\infty(\R)}\leq (1+\kappa)\|q_{0}\|_{\sL^{\infty}(\R)}\label{eq:maximum_principle}
\end{equation}
as claimed.
Furthermore, as $C_\eta  \rightarrow 0 $ for $\eta \rightarrow 0$, the time interval becomes arbitrary large. This becomes evident as \(c_{\eta}\in \mathcal{O}(\eta^{-1})\) according to \cref{prop:bounds_cgamma} (aggressively forward referencing) and \(0\leq \gamma(\delta)<\gamma(0)=1\). To detail this, it holds (estimating \(c_{\eta}\) from above and below as in \cref{prop:bounds_cgamma}):
\begin{align*}
 \lim_{\eta\rightarrow 0} \tfrac{1}{\eta^{2}}\gamma(\delta)^{\frac{1}{\eta}}=\lim_{\eta\rightarrow 0}\tfrac{\tfrac{1}{\eta^{2}}}{\gamma(\delta)^{-\frac{1}{\eta}}} = \lim_{\eta\rightarrow 0}\tfrac{-\tfrac{2}{\eta^{3}}}{\tfrac{1}{\eta^2}\ln(\gamma(\delta))\gamma(\delta)^{-\frac{1}{\eta}}} = \lim_{\eta\rightarrow 0}\tfrac{\tfrac{2}{\eta^{2}}}{\tfrac{1}{\eta^2}\ln(\gamma(\delta))^2\gamma(\delta)^{-\frac{1}{\eta}}} %&= \lim_{\eta\rightarrow 0}2\ln(\gamma(\delta))\gamma(\delta)^{\frac{1}{\eta}} 
 = 0.   
\end{align*}
The estimate in \crefrange{eq:q_t_estimate_1}{eq:q_t_estimate_3}
\[
\ \forall (t,x)\in\OT:\ |V'(W(t,x))|\leq \|V'\|_{L^{\infty}((0,2\|q_{0}\|_{\sL^{\infty}(\R)}))}
\]
needs to be justified as we do not know a priori that \(W\) and \(q\) are bounded by \(\|2q_{0}\|_{\sL^{\infty}(\R)}\). However, for a small enough time horizon, this is true, and we then obtain the estimate in \cref{eq:maximum_principle} which is, for \(\kappa<1\), stronger and implies in particular that \(\|q(t,\cdot)\|_{\sL^{\infty}(\R)}\leq 2\|q_{0}\|_{\sL^{\infty}(\R)}\). This justifies the previous estimate. A classical extension argument in time concludes the proof.
\end{proof}

\section{Surrogate nonlocal term}\label{sec:TV_uniform}

Following the idea in \cite{coclite2022general}, we aim to express our PDE in terms of a nonlocal quantity. As the exponential kernel satisfies the identity, $\gamma_\eta' \equiv -\eta \gamma'$, we can transform the nonlocal dynamics in \(q_{\eta}\) solely in the new quantity \(W[q,\exp(\tfrac{-\cdot}{\eta})]\). We thus define the following quantities:

\begin{definition}[Surrogate nonlocal term] \label{defi:surrogate}
For \(\nu\in\R_{>0}\) and \(\eta \in \R_{>0}\) we define the surrogate nonlocal term as follows:
    \begin{equation}
E_{\nu,\eta}(t,x)\coloneqq W[q_\eta,\exp(-\tfrac{\cdot}{\nu})](t,x) = \tfrac{1}{\nu}\int_{x}^{\infty}\exp\big(\tfrac{x-y}{\nu}\big)q_\eta(t,y)\dd y, \qquad (t,x)\in\OT,\label{eq:surrogate_nonlocal}
\end{equation}
where \(q_\eta\) is a solution of \cref{theo:existence_uniqueness_local}.
As the special relation \(\nu = -\eta \gamma'(0)^{-1}\) is of relevance for the prove of convergence, we define the following abbreviation:
    \begin{equation}
E_{\eta}(t,x)\coloneqq E_{-\eta \gamma'(0)^{-1},\eta}  = \tfrac{1}{-\eta \gamma'(0)^{-1}}\int_{x}^{\infty}\exp\big(\tfrac{x-y}{-\eta \gamma'(0)^{-1}}\big)q_\eta(t,y)\dd y, \qquad (t,x)\in\OT. \label{eq:surrogate_nonlocal_abbreviation}
\end{equation}
\end{definition}

As the nonlocal term \(W[q,\gamma_{\eta}]\) can then be expressed by \(W[q,\exp(\tfrac{-\cdot}{\eta})]\) as well, we aim to derive uniform \(\sTV\) bounds of the exponential nonlocal term uniformly in \(\eta\). This is made precise in the following \cref{lem:surrogate_dynamics_E}.

\begin{lemma}[The surrogate dynamics for the exponential nonlocality \(E\)]\label{lem:surrogate_dynamics_E}
Given the dynamics in \cref{defi:nonlocal_conservation_law} with solution \(q_\eta\in \sC\big([0,T];\sL^{1}(\R)\big)\), define for \(\nu\in\R_{>0}\)
 \(E_{\nu,\eta}\in \sW^{1,\infty}(\OT)\) as in \cref{defi:surrogate} and the following dynamics hold almost everywhere in \((t,x)\in\OT\):
\begin{align*}
\partial_t E_{\nu,\eta}(t,x)&=-V(W[q_\eta,\gamma_\eta](t,x))\partial_xE_{\nu,\eta}(t,x)+\tfrac{1}{\nu} V(W[q_\eta,\gamma_\eta](t,x))E_{\nu,\eta}(t,x) \\&\quad -\tfrac{1}{\nu^{2}}\!\!\int_{x}^{\infty}\!\!\!\!\e^{\frac{x-y}{\nu}}V(W[q_\eta,\gamma_\eta](t,y))\big(E_{\nu,\eta}(t,y)-\nu \partial_y E_{\nu,\eta}(t,y)\big)\dd y,\\
E_{\nu,\eta}(0,x)&=\tfrac{1}{\nu}\int_{x}^{\infty}\exp\big(\tfrac{x-y}{\nu}\big)q_{0}(y)\dd y,\ x\in\R
\end{align*}
with \(W\) being the nonlocal operator as in \cref{eqn:NBL}. 
\end{lemma}
\begin{proof}
For the sake of simplicity, we use the abbreviations $E \coloneqq E_{\nu,\eta},W \coloneqq W[q_\eta,\gamma_\eta]$ and $q \coloneqq q_\eta$. With the characteristics $\xi$ solving
\begin{align*}
    \xi(t,x;\tau) = x + \int_t^\tau V(W(s,\xi(t,x;s))) \dd s, && \forall (t,x,\tau) \in \OT \times [0,T]
\end{align*}
we obtain---as shown in \cite[Lemma 2.1]{pflug}---the following:
\begin{align*}
E(t,x)&=\tfrac{1}{\nu}\int_{x}^{\infty}\exp\big(\tfrac{x-y}{\nu}\big)q_0(\xi(t,y;0)) \partial_2 \xi(t,y;0)\dd y, && (t,x)\in\OT \\
&= \tfrac{1}{\nu}\int_{\xi(t,x;0)}^{\infty}\exp\big(\tfrac{x-\xi(0,z;t)}{\nu}\big)q_0(z) \dd z, && (t,x)\in\OT .
\end{align*}
For every \((t,x)\in\OT\), we thus obtain  
\begin{align*}
    \partial_{t} E(t,x) &=-\tfrac{1}{\nu}q_0(\xi(t,x;0)) \partial_t \xi(t,x;0) \\
    &\qquad + \tfrac{1}{\nu^2}\int_{\xi([t,x;0)}^{\infty}\exp\big(\tfrac{x-\xi(0,z;t)}{\nu}\big)q_0(z) \partial_t \xi(0,z;t) \dd z \\
    \intertext{by $\partial_t \xi(t,x;0) = - V(W(t,x)] \partial_2 \xi(t,x;0)$ as shown in \cite[Lemma 2.6, Item 1]{pflug} we obtain}
    &=\tfrac{1}{\nu} V(W(t,x))q(t,x) \\
    &\qquad -\tfrac{1}{\nu^{2}}\int_{\xi(t,x;0)}^{\infty}\exp(\tfrac{x-\xi(0,z;t)}{\nu})q_0(z) V(W(t,\xi(0,z;t)))\dd y\\
    \intertext{substituting $x = \xi(0,z;t)$ in the integral leads to}
    &=\tfrac{1}{\nu} V(W(t,x))q(t,x)-\tfrac{1}{\nu^{2}}\int_{x}^{\infty}\exp\big(\tfrac{x-y}{\nu}\big)V(W(t,x))q(t,y))\dd y
\end{align*}
which is---by using $q \equiv E - \nu \partial_2 E$---the claimed identity.
\end{proof}

\begin{remark}[Interpretation of surrogate nonlocal impact $\nu$]
The \(\nu \in \R_{>0}\) in the surrogate function \(E_{\nu}\) denotes the nonlocal scaling with the introduced exponential kernel. 

In the following convergence analysis, the choice \(\nu = -\eta \gamma'(0)\) turns out to be crucial in \cref{lem:W_E_L_infty_estimate} and \cref{theo:TV}, so, although we suggested a general other nonlocal scaling in \(\nu \), it needs to be of the order of \(\eta\in\R_{>0}\).
\end{remark}
For the entropy admissibility in the limit for $\eta \rightarrow 0$ proven in \cref{sec:convergence}, we further need a comparison between the nonlocal term $W_\eta$ and its surrogate $E_\eta$ which is specified in the following:

\begin{lemma}[\(\sL^{\infty}\) estimate between the nonlocal term \(W_{\eta}\) and its surrogate \(E_{\eta}\)]\label{lem:W_E_L_infty_estimate}
Let \(W_\eta\) and \(E_{\eta}\) be defined as in \cref{eqn:nonlocalterm} and \cref{eq:surrogate_nonlocal}, respectively, for \(\eta \in (0,0.25)\) small enough s.t. \cref{theo:max} holds for \(\epsilon \leq 1\). Then the two nonlocal terms satisfy the following \(\sL^\infty\)-estimate:
    \begin{align*}
    \|W_\eta(t,\cdot)-E_\eta(t,\cdot)\|_{\sL^\infty(\R)} \leq \eta \big(2\mathcal D_1   + 4   \big)\|q_0\|_{\sL^\infty(\R)} \quad \forall t \in [0,T].
\end{align*}
\end{lemma}
\begin{proof}
Let \((t,x)\in\OT\) be given and---recalling the identity for \(W_{\eta}\) in \cref{eqn:nonlocalterm} and the identity for the surrogate kernel \(E_{\eta}\) in \cref{eq:surrogate_nonlocal}, i.e.,
\begin{equation}
-\tfrac{\eta}{\gamma'(0)}\partial_{x}E_{\eta}\equiv E_{\nu}-q\ \implies \ q\equiv E_{\eta}+\nu\partial_{x}E_{\nu}\label{eq:E_E_x_q}
\end{equation}
where we use for the sake of simplicity the abbreviation \(\nu\coloneqq -\eta \gamma'(0)^{-1}\)
---estimate as follows:
    \begin{align*}
&|W_\eta(t,x)-E_\eta(t,x)|\\
&= \bigg|\int_{x}^\infty \!\!\gamma_\eta(x-y) \big( E_\eta(t,y) - \nu \partial_y E_\eta(t,y) \big)\dd y  - E_\eta(t,x) \bigg|\\
&= \bigg|\int_{x}^\infty \!\!\gamma_\eta(x-y) E_\eta(t,y) \dd y  - \int_{x}^\infty \!\!\gamma_\eta(x-y) \nu \partial_y E_\eta(t,y) \dd y  - E_\eta(t,x) \bigg|
\intertext{integration by parts in the second term}
&= \bigg|\int_{x}^\infty \!\!\gamma_\eta(x-y) E_\eta(t,y) \dd y + c_\eta \nu E_\eta(t,x) - \!\! \int_{x}^\infty \!\!\!\nu \gamma_\eta'(x-y)  E_\eta(t,y) \dd y  - E_\eta(t,x) \bigg|\\
&\leq \int_{x}^\infty  \big| \gamma_\eta(x-y) - \nu \gamma_\eta'(x-y)\big|  E_\eta(t,y) \dd y  + |1-c_\eta \nu|E_\eta(t,x) \\
\intertext{as \(\nu = -\eta \gamma'(0)^{-1}\) and \(0\leq E_\eta \leq 2 \|q_0\|_{\sL^\infty(\R)}\) as derived in \cref{theo:max} for \(\eta \in \R_{>0}\) small enough}
&\leq \big\| \gamma_\eta + \tfrac{ \eta }{\gamma'(0)} \gamma_\eta'\big\|_{\sL^1(\R_{>0})}  \cdot 2\|q_0\|_{\sL^\infty(\R)}  + \big|1+\tfrac{c_\eta \eta}{\gamma'(0)}\big|  \cdot 2\|q_0\|_{\sL^\infty(\R)} \\
\intertext{by the bounds on \(c_\gamma\) as derived in \cref{prop:bounds_cgamma} for the second term we obtain for \(\eta \leq \frac{1}{4}\)}
&\leq \big\| \gamma_\eta + \tfrac{\eta}{\gamma'(0)} \gamma_\eta'\big\|_{\sL^1(\R_{>0})}  \cdot 2\|q_0\|_{\sL^\infty(\R)}  + \Big|1+\tfrac{\frac{-\gamma'(0)}{\eta(1-2\eta)} \eta}{\gamma'(0)}\Big|  \cdot 2\|q_0\|_{\sL^\infty(\R)} \\
\intertext{as derived in \cref{prop:prop_gamma_deriv}, there exists \(\mathcal D_1\) s.t.\ we obtain}
&\leq \eta \big(2\mathcal D_1   + 4   \big)\|q_0\|_{\sL^\infty(\R)}.
\end{align*}
As this is uniform in \(x \in \R\), we obtain the claimed result.
\end{proof}
Next, we provide a uniform \(\sTV\) estimate of our surrogate quantity \(E_{\nu}\) when assuming that our kernel \(\gamma\) satisfies \cref{ass:gamma}. This results also in an uniform \(\sTV\) estimate for the nonlocal operator \(W_{\eta}\), as we shall see.
\begin{theorem}[\(\sTV\) estimates for surrogate model in $\eta,\nu$]\label{theo:TV}
Let the assumptions be as in \cref{ass:gamma}; then, there exists \(\mathcal{D}_{2}\in\R_{>0}\) so that one has
\begin{align}
\sup_{\eta\in\R_{>0}}|E_{\eta}(t,\cdot)|_{\sTV(\R)}&\leq |q_{0}|_{\sTV(\R)}\exp\Big(4\mD_2\|V'\|_{\sL^{\infty}((0,E_{\max}))}\|q_0\|_{\sL^\infty(\R)}t\Big),\quad \forall t\in[0,T].\label{eq:TV_non_increasing}
\end{align}
Even more, the nonlocal operator \(W_{\eta}\) is also uniformly \(\sTV\) bounded, i.e., 
\begin{align*}
|W_{\eta}(t,\cdot)|_{\sTV(\R)}%&\leq -\tfrac{\eta}{\gamma'(0)}c_{\eta}|E_\eta(t,\cdot)|_{\sTV(\R)}+\big\|\gamma_{\eta}+\tfrac{\eta}{\gamma'(0)}\gamma'_{\eta}\big\|_{\sL^{1}(\R)}|E_\eta(t,\cdot)|_{\sTV(\R)}\\
                            &\leq \big(2 +\mathcal{D}_{3}\big)|E_\eta(t,\cdot)|_{\sTV(\R)} \\
                             &\leq \big(2 +\mathcal{D}_{3}\big)|q_{0}|_{\sTV(\R)}\e^{4\mD_2\|V'\|_{\sL^{\infty}((0,E_{\max}))}\|q_0\|_{\sL^\infty(\R)}t},\ \forall t\in[0,T], \eta \in \big(0,\tfrac{1}{4}\big).
\end{align*}
\end{theorem}
\begin{proof}

By the stability result in \cref{theo:stability}, we can approximate the solution smoothly.
Thus, we can differentiate through in the equation for \(E\coloneqq E_{\eta}\) in \cref{lem:surrogate_dynamics_E} and have, 
for \((t,x)\in\OT\) (we will often write \(E_{x}\) for \(\partial_{x}E\) etc. and use \(\nu \coloneqq -\eta \gamma'(0)^{-1}\) for brevity),
\begin{align}
    \partial_{t}E_{x}&=\tfrac{1}{\nu} V'(W)W_x E-V'(W)W_xE_x-V(W)E_{xx}+\tfrac{1}{\nu^{2}}V(W)E\notag\\
    &\quad-\tfrac{1}{\nu^{3}}\int_{x}^{\infty}\exp\big(\tfrac{x-y}{\nu}\big)V(W)E\dd y+\tfrac{1}{\nu^{2}}\int_{x}^{\infty}\exp\big(\tfrac{x-y}{\nu}\big)V(W)E_y\dd y\notag\\
    &=\tfrac{1}{\nu} V'(W)W_x E-V'(W)W_xE_x-V(W)E_{xx}\label{eq:E_t_x_1}\\
    &\quad -\tfrac{1}{\nu^{2}}\int_{x}^{\infty}\exp\big(\tfrac{x-y}{\nu}\big)V'(W)\partial_{y} WE\dd y.\label{eq:E_t_x_2}
\end{align}
As we require replacement of $W$ and $W_x$ by $E$ and $E_x$, we first compute $W_x$ in terms of $E,E_x$
\begin{align}
    W_x(t,x) &= - \gamma_\eta(0) q(t,x) - \int_x^\infty \gamma_\eta'(y-x) q(t,y) \dd y\label{eq:W_x_1}\\
    \intertext{using that, thanks to the exponential kernel, we have \(q\equiv E-\nu E_{x}\)}
    &= - c_\eta \big(E(t,x)- \nu E_x(t,x) \big)- \int_x^\infty \gamma_\eta'(y-x) \big(E(t,y)- \nu E_y(t,y) \big) \dd y\notag\\
    &= \nu c_\eta E_x(t,x) + \int_{x}^\infty \big(\gamma_\eta(y-x)+ \nu \gamma_\eta'(y-x)\big) E_y(t,y)\dd y.\label{eq:W_x_2}
\end{align}
Having established this, we can now derive a uniform \(\sTV\) estimate by using the previous identity suppressing the explicit dependencies on \(t\) and (sometimes) \(x\) as well:
\begin{align*}
    &\tfrac{\dd}{\dd t}\int_{\R}|E_x|\dd x= \int_{\R}\sgn(E_x) E_{t,x}\dd x
    \intertext{taking advantage of the identity in \cref{lem:surrogate_dynamics_E}}
    &=\tfrac{1}{\nu}\int_{\R}\sgn(E_x)V'(W)W_xE\dd x-\int_{\R}\sgn(E_x)V'(W)W_xE_x\dd x\\
    &\qquad -\int_{\R}\sgn(E_x)V(W) E_{xx}\dd x-\tfrac{1}{\nu^{2}}\int_{\R}\sgn(E_x)\int_{x}^{\infty}\e^{\frac{x-y}{\nu}}V'(W) W_{y}E\dd y\dd x
    \intertext{integrating by parts}
    &=\tfrac{1}{\nu}\int_{\R}\sgn(E_x)V'(W)W_xE\dd x-\tfrac{1}{\nu^{2}}\int_{\R}\sgn(E_x)\int_{x}^{\infty}\e^{\frac{x-y}{\nu}}V'(W)W_{y} E\dd y\dd x
    \intertext{using \crefrange{eq:W_x_1}{eq:W_x_2}}
    &=\tfrac{1}{\nu}\!\!\int_{\R}\!\!\sgn(E_x(x))V'(W(x))\Big(\nu c_\eta E_x(x) + \!\!\!\int_{x}^\infty \!\!\!\!\!(\gamma_\eta(y-x)+ \nu \gamma_\eta'(y-x)) E_y(y)\dd y\Big)E(x)\dd x\\
    &\quad -\tfrac{1}{\nu^{2}}\int_{\R}\sgn(E_x(x))\int_{x}^{\infty}\e^{\frac{x-y}{\nu}}V'(W(y)) \\
    &\qquad \qquad \qquad \cdot \bigg(\nu c_\eta E_y(y) + \int_{y}^\infty \big(\gamma_\eta(z-y)+ \nu \gamma_\eta'(z-y)\big) E_z(z)\dd z\bigg)E(y)\dd y\dd x\\
    %%%%
    &=c_\eta\int_{\R}\sgn(E_x(x))V'(W(x))  E_{x}(x) E(x) \dd x\\
    &\quad +\tfrac{1}{\nu}\int_{\R}\sgn(E_x(x))V'(W(x))\int_{x}^\infty \big(\gamma_\eta(y-x)+ \nu \gamma_\eta'(y-x)\big) E_y(y)\dd y E(x)\dd x\\
    &\quad -\tfrac{c_\eta}{\nu}\int_{\R}\sgn(E_x(x))\int_{x}^{\infty}\e^{\frac{x-y}{\nu}}V'(W(y)) E_y(y) E(y)\dd y\dd x\\
    &\quad -\tfrac{1}{\nu^{2}}\!\!\!\int_{\R}\!\!\sgn(E_x(x))\!\!\!\int_{x}^{\infty}\!\!\!\e^{\frac{x-y}{\nu}}V'(W(y))\!\!\!\int_{y}^\infty \!\!\big(\gamma_\eta(z-y)+ \nu \gamma_\eta'(z-y)\big)E_z(z)\dd z E(y)\dd y\dd x\\
    %%%%%
\intertext{and changing the order of integration in the second last term yields}
    &=c_\gamma\int_{\R}V'(W(x)) |E_x(x)|  E(x) \dd x\\
    &\quad +\tfrac{1}{\nu}\int_{\R}E_y(y)\int_{-\infty}^y \sgn(E_x(x))V'(W(x))\big(\gamma_\eta(y-x)+ \nu \gamma_\eta'(y-x)\big) E(x) \dd x \dd y\\
    &\quad -\tfrac{c_\gamma}{\nu}\int_{\R}V'(W(y)) E_y(y) E(y)\int_{-\infty}^{y}\sgn(E_x(x))\e^{\frac{x-y}{\nu}}\dd x\dd y\\
    &\quad -\tfrac{1}{\nu^{2}}\!\!\!\int_{\R}\!\int_{x}^{\infty}\!\!\!\!\!\int_{x}^z \!\!\!\sgn(E_x(x))\e^{\frac{x-y}{\nu}}V'(W(y))\big(\gamma_\eta(z-y)+ \nu \gamma_\eta'(z-y)\big)E_z(z)E(y)\dd y \dd z\dd x\\
    %%%%%
    &=c_\gamma\int_{\R}V'(W(x)) |E_x(x)|  E(x) \dd x\\
    &\quad +\tfrac{1}{\nu}\int_{\R}E_y(y)\int_{-\infty}^y \sgn(E_x(x))V'(W(x))\big(\gamma_\eta(y-x)+ \nu \gamma_\eta'(y-x)\big) E(x) \dd x \dd y\\
    &\quad -\tfrac{c_\gamma}{\nu}\int_{\R}V'(W(y)) |E_y(y)| E(y)\int_{-\infty}^{y}\e^{\frac{x-y}{\nu}}\dd x\dd y\\
    &\quad -\tfrac{1}{\nu^{2}}\!\!\!\int_{\R}\!\int_{x}^{\infty}\!\!\!\!\!\int_{x}^z \!\!\!\!\sgn(E_x(x))\e^{\frac{x-y}{\nu}}V'(W(y))\big(\gamma_\eta(z-y)+ \nu \gamma_\eta'(z-y)\big)  E_z(z)E(y)\dd y \dd z\dd x.
    %%%%%    
    \intertext{The first and third terms cancel analogously to the ``classical'' exponential term estimates as e.g., in \cite{coclite2022general}, and we obtain}
    &\leq \tfrac{1}{\nu}\int_{\R}E_y(y)\int_{-\infty}^y \sgn(E_x(x))V'(W(x))\big(\gamma_\eta(y-x)+ \nu \gamma_\eta'(y-x)\big) E(x) \dd x \dd y\\
    &\quad -\tfrac{1}{\nu^{2}}\!\!\!\int_{\R}\!\!E_z(z)\!\!\int_{-\infty}^{z}\!\int_{x}^z \!\!\!\!\sgn(E_x(x))\e^{\frac{x-y}{\nu}}V'(W(y))\big(\gamma_\eta(z-y)+ \nu \gamma_\eta'(z-y)\big) E(y)\dd y \dd x\dd z
    %%%%%
    \intertext{and by another change of the order of integration in the second term}
    &= \tfrac{1}{\nu}\int_{\R}E_{y}(y)\int_{-\infty}^y \sgn(E_x(x))V'(W(x))\big(\gamma_\eta(y-x)+ \nu \gamma_\eta'(y-x)\big) E(x) \dd x \dd y\\
    &\quad -\tfrac{1}{\nu^{2}}\int_{\R}E_y(y)\int_{-\infty}^{y}V'(W(z))\big(\gamma_\eta(y-z)+ \nu \gamma_\eta'(y-z)\big) E(z) \\
    & \qquad\qquad\qquad\qquad\qquad\qquad \cdot \int_{-\infty}^y\sgn(E_x(x)) \e^{\frac{x-z}{\nu}}\dd x \dd z\dd y\\ 
    &\leq \tfrac{2}{\nu}|E|_{\sTV(\R)}\|\gamma_\eta - \nu \gamma_\eta'\|_{\sL^1(\R_{>0})}\|V'\|_{\sL^\infty((0,2\|q_0\|_{\sL^\infty(\R)}))}\|E\|_{\sL^\infty(\R)}.
    \end{align*}
    By \cref{theo:max} $\|q\|_{\sL^\infty(\R)} \leq 2\|q_0\|_{\sL^\infty(\R)}$ and consequently $\|E\|_{\sL^\infty(\R)},\|W\|_{\sL^\infty(\R)}\leq 2\|q_0\|_{\sL^\infty(\R)}$ for $\eta$ small enough. As $\nu = -\eta \gamma'(0)^{-1}$ there exists $\mD_3 \in \R_{>0}$ uniform in $\eta$ due to \cref{eq:the_key_equation} in \cref{prop:prop_gamma_deriv} with $a \equiv -\gamma'(0)^{-1}$ s.t.
    \begin{align*}
         \tfrac{\dd}{\dd t}|E|_{\sTV(\R)}\leq \mD_3\|V'\|_{\sL^\infty((0,\|q_0\|_{\sL^\infty(\R)}))}\|q_0\|_{\sL^\infty(\R)} |E|_{\sTV(\R)}.
    \end{align*}
Applying Gronwall's inequality and reintroducing the explicit dependencies on $t$ and $\eta$ we end up with the estimate
\begin{align*}
    |E_{\eta}(t,\cdot)|_{\sTV(\R)}&\leq |E_{\eta}(0,\cdot)|_{\sTV(\R)}\e^{\mathcal D_3\|V'\|_{\sL^{\infty}((0,2\|q_0\|_{\sL^\infty(\R)))}}\|q_0\|_{\sL^\infty(\R)}t}\\
    &\leq |q_{0}|_{\sTV(\R)}\e^{\mathcal D_3\|V'\|_{\sL^{\infty}((0,2\|q_0\|_{\sL^\infty(\R)))}}\|q_0\|_{\sL^\infty(\R)}t}.
\end{align*}
As the right hand side is uniform in \(\eta\in\R_{>0}\), we obtain the claimed uniform \(\sTV\) estimate. It remains to show that also \(W_{\eta}\) admits a proper \(\sTV\) bound uniform in \(\eta\). To this end, recall \crefrange{eq:W_x_1}{eq:W_x_2}, take the absolute value and integrate over space to arrive for \(t\in[0,T]\) at
\begin{align*}
    |W_{\eta}(t,\cdot)|_{\sTV(\R)}&\leq \nu c_{\eta}|E_{\eta}(t,\cdot)|_{\sTV(\R)}+ \|\gamma_{\eta}+\nu\gamma'\|_{\sL^{1}(\R)}|E_{\eta}(t,\cdot)|_{\sTV(\R).}
    \intertext{Using \(\nu=-\tfrac{\eta}{\gamma'(0)}\) and \cref{prop:bounds_cgamma} for the bounds on \(c_{\eta}\) as well as \cref{eq:the_key_equation} in \cref{prop:bounds_cgamma} for the bounds on \(\gamma_{\eta}+\nu \gamma'_{\eta}\) when \(\eta\in (0,\tfrac{1}{4})\) yields}
    &\leq -\tfrac{\eta}{\gamma'(0)}\tfrac{-\gamma'(0)}{\eta(1-2\eta)}|E_{\eta}(t,\cdot)|_{\sTV(\R)}-\tfrac{\gamma'(0)}{\eta}\eta \mD_1=\big(\tfrac{1}{1-2\eta}+\mathcal{D}_{3}\big)|E_{\eta}(t,\cdot)|_{\sTV(\R)}
\end{align*}
with \(\mathcal{D}_{3}=-\gamma'(0)\mathcal{D}_{1}\). This concludes the proof.
\end{proof}

\begin{remark}[Uniform \(\sTV\) estimate]
    As can be seen, the obtained \(\sTV\) bound can increase exponentially in time, but is, on every finite time horizon, finite. This differs from the classical result on exponential kernels \cite{coclite2022general} where the \(\sTV\) semi-norm is diminishing over time. Concluding, one can state that we sacrificed this diminishing property for the sake of more general kernels. This is particularly counter-intuitive as the solution to the local conservation law is \(\sTV\) non-increasing.

    The classical result, i.e., a non-increasing \(\sTV\)-seminorm, can still be recovered from the estimate in \cref{eq:TV_non_increasing} as in this case (the purely exponential kernel) the constant \(\mathcal{D}_{2}\) is zero so that one indeed obtains
    \[
|E_{\eta}(t,\cdot)|_{\sTV(\R)}\leq |E_{\eta}(t,\cdot)|_{\sTV(\R)}\leq  |q_{0}|_{\sTV(\R)},\quad \forall t\in[0,T].
    \]
\end{remark}
Having derived the uniform \(\sTV\) bound in the spatial variable, one can obtain ``time-compactness'' for the surrogate quantity \(E\) and the nonlocal term \(W\) as well, as the nonlocal term satisfies the balance law in \cref{lem:surrogate_dynamics_E}.
\begin{theorem}[Compactness of {\(E_{\eta}, W_{\eta}[q_{\eta},\gamma_{\eta}]\) and convergence to a weak solution for \(E_{\eta},W[q_{\eta},\gamma_{\eta}]\) and \(q_{\eta}\)}]\label{theo:strong_convergence}
The sets
\begin{align*}
\Big\{E_{\eta}\in \sC\big([0,T];\sL^{1}_{\loc}(\R)\big): \eta\in\R_{>0},\ E_{\eta}\ \text{as in \cref{theo:TV}}\Big\}&\overset{\text{c}}{\hookrightarrow} \sC\big([0,T];\sL^{1}_{\loc}(\R)\big)\\
\Big\{W[q_{\eta},\gamma_{\eta}]\in \sC\big([0,T];\sL^{1}_{\loc}(\R)\big): \eta\in\R_{>0},\ W[q_{\eta},\gamma_{\eta}]\ \text{as in \cref{eqn:nonlocalterm}}\Big\}&\overset{\text{c}}{\hookrightarrow} \sC\big([0,T];\sL^{1}_{\loc}(\R)\big)
\end{align*}
are compactly embedded into the space \(\sC\big([0,T];\sL^{1}_{\loc}(\R)\big)\). Moreover, there exists a subsequence \(\{\eta_{k}\}_{k\in\N}\subset\R_{>0}\) and \(q^{*}\in \sC\big([0,T];\sL^{1}_{\loc}(\R)\big)\cap\sL^{\infty}((0,T);\sL^{\infty}(\R))\) so that
\begin{equation}
    \lim_{\eta\rightarrow 0}\big\|X_{\eta}-q^{*}\big\|_{\sC([0,T];\sL^{1}_{\loc}(\R))}=0\ \text{ for } X_{\eta}\in \{E_{\eta},W_{\eta},q_{\eta}\}
\end{equation}
and \(W_{\eta}\coloneqq W[q_{\eta},\gamma_{\eta}]\) as in \cref{eqn:NBL} as well as \(q_{\eta}\) the unique solution to \cref{eqn:NBL}.
\end{theorem}
\begin{proof}
The ``spatial'' compactness is an intermediate consequence of the uniform \(\sTV\) bound in \cref{theo:TV}. However, we require the compactness in \(\sC([0,T];\sL^{1}_{\loc}(\R))\), and need to invoke thus \cite[Lemma 1]{simon}. It states that a set \(F\subset \sC([0,T];B)\) is relatively compact in \(\sC([0,T];B)\) iff
\begin{itemize}
    \item \(F(t)\:\big\{f(t)\in B: f\in F\}\) is relatively compact in \(B\ \forall t\in[0,T]\).
    \item \(F\) is uniformly equi-continuous, i.e.\ 
    \[
    \!\!\forall \sigma\in\R_{>0}\ \exists \delta\in\R_{>0}\ \forall f \!\in \!F\ \forall (t_{1},t_{2})\in [0,T]^{2} \text{ with } |t_{1}-t_{2}|\leq \delta:\ \|f(t_{1})-f(t_{2})\|_{B}\leq \sigma.
    \]
\end{itemize}
We chose \(F(t)=\{E_{\eta}(t,\cdot)\in \sL^{1}(I)\) and \(B=\sL^{1}(I)\) with \(I\) and an open bounded interval. Thanks to \cref{theo:TV}, we have uniform \(\sTV\) bounds, so we can invoke the compactness of \(\sBV(I)\) in \(\sL^{1}(I)\) \cite[Theorem 13.35]{leoni} and obtain the first item.

For the second item, we estimate as follows for \((t_{1},t_{2})\in[0,T]^{2},\ t_{1}\leq t_{2}\):
\begin{align*}
    &\int_{I}|E_{\eta}(t_{1},x)-E_{\eta}(t_{2},x)|\dd x=\int_{I}\Big|\int_{t_{1}}^{t_{2}}\partial_{s}E_{\eta}(s,x)\dd s\Big|\dd x\leq \int_{t_{1}}^{t_{2}}\!\!\int_{I}|\partial_{s}E_{\eta}(s,x)|\dd x\dd s\\
    \intertext{using the surrogate dynamics in \cref{lem:surrogate_dynamics_E}}
    &\leq \int_{t_{1}}^{t_{2}}\!\!\int_{I}|V(W_{\eta}(s,x))\partial_{x}E_{\eta}(s,x)|\dd x\dd s\\
    &\quad +\int_{t_{1}}^{t_{2}}\!\!\int_{I}\bigg|\tfrac{1}{\nu} V(W[q_\eta,\gamma_\eta](t,x))E_{\nu,\eta}(t,x) \\
    &\qquad \qquad \qquad -\tfrac{1}{\nu^{2}}\!\!\int_{x}^{\infty}\!\!\!\!\e^{\frac{x-y}{\nu}}V(W[q_\eta,\gamma_\eta](t,y))\big(E_{\eta}(t,y)-\nu \partial_y E_{\eta}(t,y)\big)\dd y\bigg|\dd x
    \intertext{and an integration by parts in the second term yields}
    &\leq \|V\|_{\sL^{\infty}(\mathcal{Q})}|t_{2}-t_{1}| |E_{\eta}|_{\sL^{\infty}((0,T);\sTV(\R))}\\
    &\quad +\int_{t_{1}}^{t_{2}}\int_{I}\Big|\tfrac{1}{\nu}\int_{x}^{\infty}\e^{\frac{x-y}{\nu}}E_{\eta}(t,y)V'(W[q_{\eta},\gamma_{\eta}](t,y))\partial_{y}W[q_{\eta},\gamma_{\eta}](t,y)\dd y\Big|\dd x\\
    &\leq \|V\|_{\sL^{\infty}(\mathcal{Q})}|t_{2}-t_{1}| |E_{\eta}|_{\sL^{\infty}((0,T);\sTV(\R))}\\
    &\quad +|t_{2}-t_{1}|2\|q_{0}\|_{\sL^{\infty}(\R)}\|V'\|_{\sL^{\infty}(\mathcal{Q})} \int_{\R}\tfrac{1}{\nu}\int_{x}^{\infty}\e^{\frac{x-y}{\nu}}\big|\partial_{y}W[q_{\eta},\gamma_{\eta}](t,y)\big|\dd y\dd x\\
    \intertext{exchanging order of integration}
    &\leq \|V\|_{\sL^{\infty}(\mathcal{Q})}|t_{2}-t_{1}| |E_{\eta}|_{\sL^{\infty}((0,T);\sTV(\R))}\\
    &\quad +|t_{2}-t_{1}|2\|q_{0}\|_{\sL^{\infty}(\R)}\|V'\|_{\sL^{\infty}(\mathcal{Q})} \int_{\R}\tfrac{1}{\nu}\int_{-\infty}^{x}\e^{\frac{x-y}{\nu}}\dd x\big|\partial_{y}W[q_{\eta},\gamma_{\eta}](t,y)\big|\dd y\\
    &\leq \|V\|_{\sL^{\infty}(\mathcal{Q})}|t_{2}-t_{1}| |E_{\eta}|_{\sL^{\infty}((0,T);\sTV(\R))} \\
    &\qquad +|t_{2}-t_{1}|2\|q_{0}\|_{\sL^{\infty}(\R)}\|V'\|_{\sL^{\infty}(\mathcal{Q})} |W_{\eta}|_{\sL^{\infty}((0,T);\sTV(\R))}.
\end{align*}
Thanks to \cref{theo:TV}, all the terms on the right hand side are uniformly bounded with respect to \(\eta\), so that we obtain Lipschitz continuity in time when measuring in space in \(\sL^{1}\), and, in particular, we have satisfied item 2 (the uniform equi-continuity).
As \(I\subset\R\) being open and bounded was arbitrary, we conclude that the set \(\{E_{\eta}:\ \eta\in\R_{>0}\}\) is compact in \(\sC([0,T];\sL^{1}_{\text{loc}}(\R))\).
Concerning the compactness of \(\big\{W_{\eta}:\ \eta\in(0,\tfrac{1}{2})\big\}\) we just remark that the spatial compactness (item 1) has been established in \cref{theo:TV} as was done for \(E_{\eta}\), and, for the time compactness, one would compute \(\partial_{t}W_{\eta}\) and express everything in terms of \(E\). We do not go into the details.

Finally, concerning the convergence, we have, by the compactness of \(E_{\eta}\), that \(\exists q^{*}\in \sC\big([0,T];\sL^{1}_{\loc}(\R)\big)\) and a subsequence \((\eta_{k})_{k\in\N},\ \lim_{k\rightarrow\infty}\eta_{k}=0\) so that
\[
\lim_{k\rightarrow\infty}\|E_{\eta_{k}}-q^{*}\|_{\sC([0,T];\sL^{1}_{\loc}(\R))}=0.
\]
Thanks to the identity in \cref{eq:E_E_x_q}, we obtain, on the same subsequence,
\[
\forall t\in[0,T]:\  \eta_{k}|E_{\eta_{k}}(t,\cdot)|_{\sTV(\R)}=\|E_{\eta_{k}}(t,\cdot)-q_{\eta_{k}}(t,\cdot)\|_{\sL^{1}(\R)}.
\]
and for \(\eta_{k}\rightarrow 0\), the convergence of \(q_{\eta_{k}}\in\sC\big([0,T];\sL^{1}_{\text{loc}}(\R)\big)\). Finally, concerning the convergence of \(W_{\eta_{k}}\), we estimate for \(t\in[0,T]\) and \(\Omega\subset\R\) open and bounded
\begin{align*}
    \|q^{*}-W_{\eta_{k}}\|_{\sC([0,T];\sL^{1}(\Omega))}&\leq \|q^{*}-E_{\eta_{k}}\|_{\sC([0,T];\sL^{1}(\Omega))}+\|E_{\eta_{k}}-W_{\eta_{k}}\|_{\sC([0,T];\sL^{1}(\Omega))}\\
    &\leq \|q^{*}-E_{\eta_{k}}\|_{\sC([0,T];\sL^{1}(\Omega))}+|\Omega|\!\!\!\sup_{t\in[0,T]}\!\!\|E_{\eta_{k}}(t,\cdot)-W_{\eta_{k}}(t,\cdot)\|_{\sL^{\infty}(\R)}
    \intertext{ and taking advantage of \cref{lem:W_E_L_infty_estimate} we arrive at}
    &\leq \|q^{*}-E_{\eta_{k}}\|_{\sC([0,T];\sL^{1}(\Omega))}+|\Omega|\eta \big(2\mathcal D_1   + 4   \big)\|q_0\|_{\sL^\infty(\R)}.
\end{align*}
The first term converges to zero as previously shown, and the second one converges (\(\Omega\) was assumed to be open and bounded, and is thus Lebesgue measurable with finite Lebesgue measure \(|\Omega|\)) for any sequence \(\eta_{k}\rightarrow 0\) to zero, resulting in the claimed \(\sC\big([0,T];\sL^{1}_{\loc}(\R)\big)\) convergence of \(W_{\eta_{k}}\). This concludes the proof.
\end{proof}

\section{Convergence to the entropy solution}\label{sec:convergence}
In this section, we show that if the nonlocal solution converges strongly in \(\sL^{1}\), that it is automatically entropic. For this, we follow an approach by \cite{colombo2023nonlocal} for the surrogate nonlocal term \(E_\eta\) as defined in \cref{eq:surrogate_nonlocal}.

Together with  \cref{sec:TV_uniform}, this gives the convergence of solutions of the nonlocal conservation law to the local entropy solution for the nonlocal kernels defined in \cref{ass:gamma}.
\begin{theorem}[Convergence to the Entropy solution]\label{theo:convergence}
For \(\eta\in\R_{>0}\) and \(T\in\R_{>0}\) let \(q_{\eta}\in \sC([0,T];\sL^{1}_{\loc}(\R))\cap \sL^{\infty}((0,T);\sL^{\infty}(\R))\) be the solution to the nonlocal conservation law stated in \cref{defi:nonlocal_conservation_law}. Assume that there exists \(q^{*}\in \sC\big([0,T];\sL^{1}_{\loc}(\R)\big)\cap \sL^{\infty}((0,T);\sL^{\infty}(\R))\) so that
\[
\lim_{\eta\rightarrow 0}\|E_{\eta}-q^{*}\|_{\sC([0,T];\sL^{1}_{\loc}(\R))}=0
\]
where \(E_{\eta}\coloneqq E[q_{\eta},\gamma_{\eta}]\) as in \cref{theo:TV}.
Then \(q^{*}\) is the weak entropy solution \cref{defi:entropy} of the local conservation law \cref{eq:local_conservation_law}.
\end{theorem}
\begin{proof}
We follow \cite{colombo2023nonlocal} and show the entropy admissibility for \(E\). To this end, let \((\alpha,\beta)\in \sC^{2}(\R)\times \sC^{2}(\R)\) as in \cref{defi:entropy} be entropy flux pairs such that
\begin{equation}
\alpha''\geqq 0\ \wedge\ \beta'\equiv \alpha'(\cdot)\big(V(\cdot)+(\cdot) V'(\cdot)\big).\label{eq:entropy_flux_pair}
\end{equation}
Let \(\phi\in C^{\infty}_{\text{c}}(\R;\R_{\geq0})\) be given. We need to show that (see \cref{defi:entropy})
\begin{equation}
\mathcal{E}[\alpha,\phi,E_{\eta}]\coloneqq \int_{\OT}\alpha(E_{\eta})\phi_{t}+\beta(E_{\eta})\phi_{x}\dd x\dd t+\int_{\R}\alpha(E_{\eta}(0,x))\phi(0,x)\dd x\geq 0.\label{eq:entropy}
\end{equation}
As \(q\mapsto \mathcal{E}[\alpha,\phi,q]\) is continuous for \(q\in \sC([0,T];\sL^{1}_{\loc}(\R))\cap\sL^{\infty}((0,T);\sL^{\infty}(\R))\), we can invoke \cref{theo:stability} and can prove the inequality for sufficiently smooth \(E_{\eta}\). To this end, manipulate as follows using the abbreviation \(\nu \coloneqq -\eta \gamma'(0)^{-1}\)
\begin{align}
    &\mathcal{E}[\alpha,\phi,E_{\eta}]\notag\\
    &=-\int_{\OT}\alpha'(E_{\eta})\partial_{t}E_{\eta}\phi+\beta'(E_{\eta})\partial_{x}E_{\eta}\phi\dd x\dd t\notag\\
    \intertext{using \cref{eq:entropy_flux_pair} on \(\beta'\)}
    &=-\int_{\OT}\alpha'(E_{\eta})\phi\Big(\partial_{t}E_{\eta}+\big(V(E_{\eta})+E_{\eta}V'(E_{\eta})\big)\partial_{x}E_{\eta}\Big)\dd x\dd t\notag
    \intertext{and inserting the surrogate identity in \cref{lem:surrogate_dynamics_E}}
    &=-\int_{\OT}\alpha'(E_{\eta})\phi\bigg(V(E_{\eta})\partial_{x}E_{\eta}+V'(E_{\eta})E_{\eta}\partial_{x}E_{\eta}-V(W_{\eta})\partial_{x}E_{\eta}+\tfrac{1}{\nu}V(W_{\eta})E_{\eta}\notag\\
    &\qquad -\tfrac{1}{\nu^{2}}\int_{x}^{\infty}\e^{\frac{x-y}{\nu}}V(W(t,y))\big(E_{\eta}(t,y)-\nu\partial_{y}E_{\eta}(t,y)\big)\dd y\bigg)\dd x\notag\\
    &=\int_{\OT}\!\!\!\!\!\alpha'(E_{\eta})\phi\tfrac{\dd}{\dd x}\Big(-V(E_{\eta})E_{\eta}+\tfrac{1}{\nu}\!\!\int_{x}^{\infty}\!\!\!\!\!\e^{\frac{x-y}{\nu}}V(W(t,y))\big(E_{\eta}(t,y)-\nu\partial_{y}E_{\eta}(t,y)\big)\dd y\Big)\dd x\notag\\
    &=\int_{\OT}\!\!\!\!\!\partial_{x}\big(\alpha'(E_{\eta})\phi\big)\bigg(V(E_{\eta})E_{\eta}-\tfrac{1}{\nu}\!\!\int_{x}^{\infty}\!\!\!\!\!\e^{\frac{x-y}{\nu}}V(W(t,y))\big(E_{\eta}(t,y)-\nu\partial_{y}E_{\eta}(t,y)\big)\dd y\bigg)\dd x\notag
    \intertext{and as according to \cref{eq:E_E_x_q} we have \(q\equiv E_{\eta}-\nu\partial_{x}E_{\eta}\)}
    &=\int_{\OT}\!\!\!\!\!\partial_{x}\big(\alpha'(E_{\eta})\phi\big)\bigg(V(E_{\eta})E_{\eta}-\tfrac{1}{\nu}\int_{x}^{\infty}\e^{\frac{x-y}{\nu}}V(W(t,y))q(t,y)\dd y\bigg)\dd x\notag\\
    &= \iint_\OT \Big(V(E_{\eta})E_{\eta}  -  \tfrac{1}{\nu}\int_{x}^\infty \e^{\frac{x-y}{\nu}}V(W_\eta(t,y))q(t,y) \dd y \Big)\partial_x \big(\alpha'(E_{\eta})\phi\big)  \dd x \dd t \notag\\
    &= \iint_\OT \Big(V(E_{\eta})E_{\eta}  -  \tfrac{1}{\nu}\int_{x}^\infty \e^{\frac{x-y}{\nu}}V(W_\eta(t,y))q(t,y) \dd y \Big)\alpha'(E_{\eta})\phi_x   \dd x \dd t  \label{eq:entropy21}\\
    &\qquad + \iint_\OT \Big(V(E_{\eta})E_{\eta}  -  \tfrac{1}{\nu}\int_{x}^\infty \e^{\frac{x-y}{\nu}}V(W_\eta(t,y))q(t,y) \dd y \Big)\partial_x \big(\alpha'(E_{\eta})\big) \phi  \dd x \dd t. \label{eq:entropy22} 
\end{align}
\Cref{eq:entropy21} converges to zero for \(\eta\rightarrow 0\), as we will detail in the following.
It holds for \((t,x)\in\OT\)
\begin{align*}
    &V(E_{\eta}(t,x))E_{\eta}(t,x)-\tfrac{1}{\nu}\int_{x}^{\infty}\e^{\frac{x-y}{\nu}}V(W_{\eta}(t,y))q(t,y)\dd y\\
    &= V(E_{\eta}(t,x))E_{\eta}(t,x)-\tfrac{1}{\nu}\int_{x}^{\infty}\e^{\frac{x-y}{\nu}}V(E_{\eta}(t,y))q(t,y)\dd y\\
    &\quad +\tfrac{1}{\nu}\int_{x}^{\infty}\e^{\frac{x-y}{\nu}}\big(V(E_{\eta}(t,y))-V(W_{\eta}(t,y))\big)q(t,y)\dd y\\
    &= V(E_{\eta}(t,x))E_{\eta}(t,x)-\tfrac{1}{\nu}\int_{x}^{\infty}\e^{\frac{x-y}{\nu}}V(E_{\eta}(t,y))E_{\eta}(t,y)\dd y\\
    &\quad +\int_{x}^{\infty}\e^{\frac{x-y}{\nu}}V(E_{\eta}(t,y))E_{\nu,y}(t,y)\dd y\\
    &\quad +\tfrac{1}{\nu}\int_{x}^{\infty}\e^{\frac{x-y}{\nu}}\big(V(E_{\eta}(t,y))-V(W_{\eta}(t,y))\big)q(t,y)\dd y\\
    \intertext{and an integration by parts in the third term}
    &=-\int_{x}^{\infty}\e^{\frac{x-y}{\nu}}V'(E_{\eta}(t,y))E_{\eta}E_{\nu,y}(t,y)\dd y\\
    &\quad +\tfrac{1}{\nu}\int_{x}^{\infty}\e^{\frac{x-y}{\nu}}\big(V(E_{\eta}(t,y))-V(W_{\eta}(t,y))\big)q(t,y)\dd y.
    \intertext{Thus, we have when introducing \(\mathcal{Q}=\big(0,2\|q_{0}\|_{\sL^{\infty}(\R)}\big)\) and estimating \(\|E_{\eta}\|_{\sL^{\infty}(\OT)}\leq 2\|q_{0}\|_{\sL^{\infty}(\R)}\) which is according to \cref{theo:max} possible}
    \eqref{eq:entropy21}&\leq \|\alpha'\|_{\sL^{\infty}(\mathcal{Q})}\|\phi_{x}\|_{\sL^{\infty}(\OT)}\Big\|V(E_{\eta})E_{\eta}\!-\!\tfrac{1}{\nu}\!\int_{\cdot}^{\infty}\!\!\!\!\e^{\frac{\cdot-y}{\nu}}V(W_{\eta}(t,y))q(t,y)\dd y\Big\|_{\sL^{1}(\supp(\phi))}\\
     &\leq \|\alpha'\|_{\sL^{\infty}(\mathcal{Q})}\|\phi_{x}\|_{\sL^{\infty}(\OT)}\bigg(\Big\|\int_{\cdot}^{\infty}\e^{\frac{\cdot-y}{\nu}}V'(E_{\eta}(t,y))E_{\eta}E_{\nu,y}(t,y)\dd y\Big\|_{\sL^{1}(\OT)}\\
     &\qquad \qquad +\Big\|\tfrac{1}{\nu}\!\!\!\int_{\cdot}^{\infty}\!\!\!\!\e^{\frac{\cdot-y}{\nu}}\big(V(E_{\eta}(t,y))-V(W_{\eta}(t,y))\big)q(t,y)\dd y\Big\|_{\sL^{1}(\supp(\phi))} \bigg)\\
     &\leq \nu\|\alpha'\|_{\sL^{\infty}(\mathcal{Q})}\|\phi_{x}\|_{\sL^{\infty}(\OT)}\|V'\|_{\sL^{\infty}(\mathcal{Q})}2\|q_{0}\|_{\sL^{\infty}(\R)}|E_{\eta}|_{\sL^{1}((0,T);\sTV(\R))}\\
     &\quad +\|\alpha'\|_{\sL^{\infty}(\mathcal{Q})}\|\phi_{x}\|_{\sL^{\infty}(\OT)}\|V'\|_{\sL^{\infty}(\mathcal{Q})}2\|q_{0}\|_{\sL^{\infty}(\R)}\|E_{\eta}\!-\!W_{\eta}\|_{\sL^{\infty}((0,T);\sL^{\infty}(\R))}|\supp(\phi)|.
\end{align*}
The first term converges to zero for \(\eta\rightarrow 0\) recalling the uniform \(\sTV\) bounds on \(E_{\eta}\) in \cref{theo:TV}, the second term converges to zero according to \cref{lem:W_E_L_infty_estimate}.

We continue with \cref{eq:entropy22} for which we need to show that it is nonnegative. We obtain
\begin{align}
\eqref{eq:entropy22}&=\iint_\OT\!\! \Big(V(E_{\eta})E_{\eta}  -  \tfrac{1}{\nu}\int_{x}^\infty\!\! \e^{\frac{x-y}{\nu}}V(W_\eta(t,y))q(t,y) \dd y \Big)\partial_x \alpha'(E_{\eta}) \phi  \dd x \dd t \notag\\
&=\tfrac{1}{\nu}\!\!\! \iint_\OT \!\int_{x}^\infty \!\!\! \e^{\frac{x-y}{\nu}} \big(V(E_{\eta}(t,x))  -   V(W_\eta(t,y))  \big)q(t,y)\dd y\, \phi\partial_x \alpha'(E_{\eta})   \dd x \dd t \notag\\
&= \tfrac{1}{\nu}\iint_\OT q(t,y)\int_{-\infty}^y \e^{\frac{x-y}{\nu}} V(E_{\eta}(t,x)) \phi(t,x) \partial_x \alpha'(E_{\eta}(t,x))   \dd x \dd y\dd t \notag\\
&\quad -\tfrac{1}{\nu}\iint_\OT  V(W_\eta(t,y))  q(t,y)\int_{-\infty}^y \e^{\frac{x-y}{\nu}}  \phi(t,x) \partial_x \alpha'(E_{\eta}(t,x))   \dd x \dd y\dd t \notag
\intertext{with $\Psi(u) \coloneqq \int_0^u \alpha''(s)V(s)\dd s$ we obtain}
&= \tfrac{1}{\nu}\iint_\OT q(t,y)\int_{-\infty}^y \e^{\frac{x-y}{\nu}}  \phi(t,x) \partial_x \Psi(E_{\eta}(t,x))   \dd x \dd y\dd t \label{eq:psi}\\
&\quad -\tfrac{1}{\nu}\iint_\OT  V(W_\eta(t,y))  q(t,y)\int_{-\infty}^y \e^{\frac{x-y}{\nu}}  \phi(t,x) \partial_x \alpha'(E_{\eta}(t,x))   \dd x \dd y\dd t \notag
\intertext{integration by parts in both terms yields}
&=  \tfrac{1}{\nu}\iint_\OT q(t,y)\int_{-\infty}^y \e^{\frac{x-y}{\nu}}   \phi_x(t,x) \big(\Psi(E_{\eta}(t,y)) - \Psi(E_{\eta}(t,x))   \dd x \dd y\dd t \notag\\
&\quad +\tfrac{1}{\nu^2}\iint_\OT q(t,y)\int_{-\infty}^y \e^{\frac{x-y}{\nu}}   \phi(t,x) \big(\Psi(E_{\eta}(t,y)) - \Psi(E_{\eta}(t,x))   \dd x \dd y\dd t\notag \\
&\quad +\tfrac{1}{\nu}\iint_\OT  V(W_\eta)  q\int_{-\infty}^y \e^{\frac{x-y}{\nu}}  \phi_x(t,x)\big( \alpha'(E_{\eta}(t,y)) - \alpha'(E_{\eta}(t,x)) \big)  \dd x \dd y\dd t\notag\\
&\quad +\tfrac{1}{\nu^2}\iint_\OT  V(W_\eta)  q\int_{-\infty}^y \e^{\frac{x-y}{\nu}}  \phi(t,x)\big( \alpha'(E_{\eta}(t,y)) - \alpha'(E_{\eta}(t,x)) \big)  \dd x \dd y\dd t\notag\\
&= \tfrac{1}{\nu}\iint_\OT q(t,y)\int_{-\infty}^y \e^{\frac{x-y}{\nu}}   \phi_x(t,x) \big(\Psi(E_{\eta}(t,y)) - \Psi(E_{\eta}(t,x)) \big)  \dd x \dd y\dd t \label{eq:42_1}\\
&\quad +\tfrac{1}{\nu}\!\!\iint_\OT\!\!\!  V(W_\eta)  q\int_{-\infty}^y \e^{\frac{x-y}{\nu}}  \phi_x(t,x)\big( \alpha'(E_{\eta}(t,y)) - \alpha'(E_{\eta}(t,x)) \big)  \dd x \dd y\dd t\label{eq:42_2}\\
&\quad +\tfrac{1}{\nu^2}\iint_\OT  q(t,y)\int_{-\infty}^y \e^{\frac{x-y}{\nu}}  \phi(t,x)\mathcal H\big(E_{\eta}(t,y),E_{\eta}(t,x)\big) \dd x \dd y\dd t\label{eq:42_3}\\
&\quad +\!\tfrac{1}{\nu^2}\!\!\!    \iint_\OT\!\!\!\!\!  \big(V(W_\eta)-V(E_{\eta}) \big)  q\!\!\!\int_{-\infty}^y \!\!\!\!\!\!\e^{\frac{x-y}{\nu}}  \phi(t,x)\big( \alpha'(E_{\eta}(t,y)) - \alpha'(E_{\eta}(t,x)) \big)  \dd x \dd y\dd t\label{eq:42_4}
\end{align}
with 
\[
\mathcal H:\begin{cases}\R^{2}&\rightarrow\R\\
(a,b)&\mapsto  \Psi(a)-\Psi(b) - V(a)(\alpha'(a)-\alpha'(b)),
\end{cases}
\]
with \(\Psi\) before \cref{eq:psi}.
\cref{eq:42_3} is nonnegative (see \cite[p.18]{colombo2023nonlocal}), the rather standard argument consists of computing the derivative with regard to \(b\) and using that,  \(\mathcal H(a,a) = 0 \ \forall a \in \R, V'\leqq 0\) and \(\alpha''\geq 0\), in equations
\begin{align*}
    \partial_b \mathcal H(a,b) &= \alpha''(b) \big( V(a)-V(b) \big)  \begin{cases}
        \leq 0 & b < a \\
        \geq 0 & b > a
    \end{cases}
\end{align*}
consequently, \(\mathcal H(a,b)\) is monotonically decreasing w.r.t. \(b\) until \(b=a\), and monotonically increasing beyond, i.e. for \(b>a\).

Thus, we need to show only that \crefrange{eq:42_1}{eq:42_2} and \cref{eq:42_4} converge to zero for \(\eta\rightarrow 0\)---keeping in mind that \(\nu=\tfrac{-\eta}{\gamma'(0)}\)---as in \cref{theo:TV}.
Thus, we manipulate as follows
\begin{align*}
    |\eqref{eq:42_1}|&\leq 2\|q_{0}\|_{\sL^{\infty}(\R)}\|\phi_{x}\|_{\sL^{\infty}(\OT)}\|\Psi'\|_{\sL^{\infty}(\mathcal{Q})}\tfrac{1}{\nu}\iint_{\OT}\int_{-\infty}^{y}\!\!\!\!\e^{\frac{x-y}{\nu}}|E_{\eta}(t,y)-E_{\eta}(t,x)|\dd x\dd y\dd y\\
    &=2\|q_{0}\|_{\sL^{\infty}(\R)}\|\phi_{x}\|_{\sL^{\infty}(\OT)}\|\Psi'\|_{\sL^{\infty}(\mathcal{Q})}\tfrac{1}{\nu}\iint_{\OT}\int_{-\infty}^{0}\!\!\!\!\!\!\!\e^{\frac{z}{\nu}}|E_{\eta}(t,y)-E_{\eta}(t,z+y)|\dd z\dd y\dd y\\
    &= 2\|q_{0}\|_{\sL^{\infty}(\R)}\|\phi_{x}\|_{\sL^{\infty}(\OT)}\|\Psi'\|_{\sL^{\infty}(\mathcal{Q})}\tfrac{1}{\nu}\int_{0}^{T}\!\!\int_{\R_{\leq 0}}\!\!\!\!\!\!\e^{\frac{z}{\nu}}\!\!\int_{\R}|E_{\eta}(t,y)-E_{\eta}(t,z+y)|\dd y\dd z\dd t
    \intertext{where we use a classical characterization of \(TV\) functions, namely \cite[Theorem 13.48]{leoni}}
    &\leq 2\|q_{0}\|_{\sL^{\infty}(\R)}\|\phi_{x}\|_{\sL^{\infty}(\OT)}\|\Psi'\|_{\sL^{\infty}(\mathcal{Q})}\tfrac{1}{\nu}\int_{0}^{T}|E_{\eta}(t,\cdot)|_{\sTV(\R)}\int_{\R_{< 0}}(-z)\e^{\frac{z}{\nu}}\dd z\dd t\\
    &=2\nu\|q_{0}\|_{\sL^{\infty}(\R)}\|\phi_{x}\|_{\sL^{\infty}(\OT)}\|\Psi'\|_{\sL^{\infty}(\mathcal{Q})}T|E_{\eta}|_{\sL^{\infty}((0,T);\sTV(\R))}   \overset{\eta\rightarrow 0}{\longrightarrow} 0.
\end{align*}
This last convergence to zero becomes clear when recalling that all terms are bounded and invariant with regard to \(\eta\) except for \(\eta \gamma'(0)^{-1}|E_{\eta}|_{\sL^{\infty}((0,T);\sTV(\R))}\). Thanks to \cref{theo:TV}, this last term, however, also converges to zero for \(\eta\rightarrow 0\) concluding the claim on that \cref{eq:42_1} converges to zero for \(\eta\rightarrow 0\). 

Almost the same argumentation can be made for the term \eqref{eq:42_2}, we do not detail it further. 
Eventually, we look at the last remaining term, \eqref{eq:42_4}, and manipulate as follows
\begin{align*}
    |\eqref{eq:42_4}|&=\Big|\tfrac{1}{\nu}\iint_\OT  \big(V(W_\eta(t,y))-V(E_{\eta}(t,y)) \big)  q(t,y)\\
    &\qquad\qquad \cdot\int_{-\infty}^y \e^{\frac{x-y}{\nu}}  \partial_{x}\phi(t,x)\big( \alpha'(E_{\eta}(t,y)) - \alpha'(E_{\eta}(t,x)) \big)  \dd x \dd y\dd t\Big|\\
    &\quad +\Big|\tfrac{1}{\nu}\iint_\OT  \big(V(W_\eta(t,y))-V(E_{\eta}(t,y)) \big)  q(t,y)\\
    &\qquad\qquad\cdot\int_{-\infty}^y \e^{\frac{x-y}{\nu}}  \phi(t,x)\alpha''(E_{\eta}(t,x))\partial_{x}E_{\eta}(t,x) \dd x \dd y\dd t\Big|\\
    &\leq \|V'\|_{\sL^{\infty}((0,2\|q_{0}\|_{\sL^{\infty}(\R)}))}\|W_{\eta}(t,\cdot)-E_{\eta}(t,\cdot)\|_{\sL^{\infty}(\R)}4\|q_{0}\|_{\sL^{\infty}(\R)}\\
    &\qquad\qquad \cdot\|\alpha'\|_{\sL^{\infty}((0,2\|q_{0}\|_{\sL^{\infty}(\R)}))}\tfrac{1}{\nu}\iint_{\OT}\int_{-\infty}^{y}\e^{\frac{x-y}{\nu}} |\partial_{2}\phi(t,x)|\dd x \dd y\dd t\\
    &\quad +\|V'\|_{\sL^{\infty}((0,2\|q_{0}\|_{\sL^{\infty}(\R)}))}\|W_{\eta}(t,\cdot)-E_{\eta}(t,\cdot)\|_{\sL^{\infty}(\R)}2\|q_{0}\|_{\sL^{\infty}(\R)}\\
    &\qquad \qquad \cdot\|\phi\|_{\sL^{\infty}(\OT)}\tfrac{1}{\nu}\iint_{\OT}\int_{-\infty}^{y}\e^{\frac{x-y}{\nu}}  \alpha''(E_{\eta}(t,x))|\partial_{x}E_{\eta}(t,x)| \dd x \dd y\dd t\\
    &\leq \|V'\|_{\sL^{\infty}((0,2\|q_{0}\|_{\sL^{\infty}(\R)}))}\|W_{\eta}(t,\cdot)-E_{\eta}(t,\cdot)\|_{\sL^{\infty}(\R)}4\|q_{0}\|_{\sL^{\infty}(\R)}\\
    &\qquad \qquad \cdot\|\alpha'\|_{\sL^{\infty}((0,2\|q_{0}\|_{\sL^{\infty}(\R)}))}\tfrac{1}{\nu}\iint_{\OT}|\partial_{2}\phi(t,x)|\int_{-\infty}^{y}\e^{\frac{x-y}{\nu}} \dd y \dd x\dd t\\
    &\quad +\|V'\|_{\sL^{\infty}((0,2\|q_{0}\|_{\sL^{\infty}(\R)}))}\|W_{\eta}(t,\cdot)-E_{\eta}(t,\cdot)\|_{\sL^{\infty}(\R)}2\|q_{0}\|_{\sL^{\infty}(\R)}\\
    &\qquad\qquad \cdot\|\phi\|_{\sL^{\infty}(\OT)}\|\alpha''\|_{\sL^{\infty}((0,2\|q_{0}\|_{\sL^{\infty}(\R)}))}\tfrac{1}{\nu}\iint_{\OT}|\partial_{x}E_{\eta}(t,x)\int_{x}^{\infty}\e^{\frac{x-y}{\nu}}   \dd y \dd x\dd t\\
    &\leq \|V'\|_{\sL^{\infty}((0,2\|q_{0}\|_{\sL^{\infty}(\R)}))}\|W_{\eta}(t,\cdot)-E_{\eta}(t,\cdot)\|_{\sL^{\infty}(\R)}4\|q_{0}\|_{\sL^{\infty}(\R)}\|\partial_{2}\phi\|_{\sL^{1}(\supp(\phi))}\\
    &\quad +\|V'\|_{\sL^{\infty}((0,2\|q_{0}\|_{\sL^{\infty}(\R)}))}\|W_{\eta}(t,\cdot)-E_{\eta}(t,\cdot)\|_{\sL^{\infty}(\R)}2\|q_{0}\|_{\sL^{\infty}(\R)} \\
    &\qquad\qquad \cdot \|\phi\|_{\sL^{\infty}(\OT)}\|\alpha''\|_{\sL^{\infty}((0,2\|q_{0}\|_{\sL^{\infty}(\R))})}T|E_{\eta}(t,\cdot)|_{\sL^{\infty}((0,T);\sTV(\R))}.
\end{align*}
As can be seen, all previous terms are bounded uniformly in \(\eta\) respectively except for
\[
\|W_{\eta}(t,\cdot)-E_{\eta}(t,\cdot)\|_{\sL^{\infty}(\R)}\ \text{ and }\ |E_{\eta}(t,\cdot)|_{\sL^{\infty}((0,T);\sTV(\R))},\ t\in[0,T].
\]
However, the first term converges to zero according to \cref{lem:W_E_L_infty_estimate}, and the second term is uniformly bounded thanks to \cref{theo:TV}. Thus, \(|\eqref{eq:42_4}|\) converges to zero, concluding the proof.
\end{proof}
The following final and main theorem summarizes the obtained results:
\begin{theorem}[Convergence to the local Entropy solution]
    Let \cref{ass:initialdatum_velocity} and \cref{ass:gamma} hold and let \(q_{\eta}\in \sC\big([0,T];\sL^{1}_{\loc}(\R)\big)\) be the unique weak solution to \cref{defi:nonlocal_conservation_law}. Then, it holds (for every sequence) \(\{\eta_{k}\}_{k\in\N}\subset\R_{>0},\ \lim_{k\rightarrow\infty}\eta_{k}=0\) that
    \begin{align*}
        \lim_{k\rightarrow\infty}\big\|q_{\eta_{k}}-q^{*}\big\|_{\sC([0,T];\sL^{1}_{\loc}(\R))}&=0\\
        \lim_{k\rightarrow\infty}\big\|W[q_{\eta_k},\gamma_{\eta_k}]-q^{*}\big\|_{\sC([0,T];\sL^{1}_{\loc}(\R))}&=0\\
        \lim_{k\rightarrow\infty}\big\|E_{\eta_k}-q^{*}\big\|_{\sC([0,T];\sL^{1}_{\loc}(\R))}&=0
    \end{align*}
    and \(q^{*}\in\sC\big([0,T];\sL^{1}_{\loc}(\R)\big)\) is the unique entropy solution as in \cref{defi:entropy} to the local conservation law \cref{defi:local_conservation_law}. In other words, the solution to the nonlocal conservation law in \cref{defi:nonlocal_conservation_law}, the nonlocal quantity \(W_{\eta}\) and the surrogate \(E_{\eta}\) as in \cref{theo:strong_convergence} converge strongly in \(\sC(\sL^{1}_{\loc})\) to the corresponding local entropy solutions.
\end{theorem}
\begin{proof}
    The proof for subsequences consists of combining \cref{theo:strong_convergence} and \cref{theo:convergence}. The convergence on any sequence \(\eta\rightarrow 0\) follows by the uniqueness of the local entropy solution as stated in \cref{theo:existence_uniqueness_local}.
\end{proof}

\begin{figure}
    \centering
    \pgfplotstableread{data/convergence1_eta1.txt}{\etaAm}
    \pgfplotstableread{data/convergence1_eta2.txt}{\etaBm}
    \pgfplotstableread{data/convergence1_eta3.txt}{\etaCm}
    \pgfplotstableread{data/convergence1_eta4.txt}{\etaDm}
    \pgfplotstableread{data/convergence1_eta5.txt}{\etaEm}
    \pgfplotstableread{data/convergence1_eta6.txt}{\etaFm}

    \pgfplotstableread{data/convergence2_eta1.txt}{\etaAe}
    \pgfplotstableread{data/convergence2_eta2.txt}{\etaBe}
    \pgfplotstableread{data/convergence2_eta3.txt}{\etaCe}
    \pgfplotstableread{data/convergence2_eta4.txt}{\etaDe}
    \pgfplotstableread{data/convergence2_eta5.txt}{\etaEe}
    \pgfplotstableread{data/convergence2_eta6.txt}{\etaFe}
    
\begin{tikzpicture}
\begin{axis}[ymajorgrids,
grid style=dotted,height=5cm,width=6cm,
          xmin=-1,ymin=0,
          xmax=4,ymax=1.2,
          restrict y to domain=0:1.2,
          restrict x to domain=-1:4,
          xtick={-1,0,1,2,3,4},
          ytick={0.0,0.2,0.4,0.6,0.8,1.0,1.2},xlabel={$x$},
          ylabel={\(q_\eta(t,\cdot)\)},legend style={draw=none,fill=none,font=\plotfont,anchor=north west, at={(1,1)}},legend cell align={left},tick label style={font=\plotfont}, label style={font=\plotfont}]

\addplot[col1!0!col2] table [x = {x}, y = {q}] {\etaAm};
\addplot[col1!20!col2] table [x = {x}, y = {q}] {\etaBm};
\addplot[col1!40!col2] table [x = {x}, y = {q}] {\etaCm};
\addplot[col1!60!col2] table [x = {x}, y = {q}] {\etaDm};
\addplot[col1!80!col2] table [x = {x}, y = {q}] {\etaEm};
\addplot[col1!100!col2] table [x = {x}, y = {q}] {\etaFm};

\end{axis}
\end{tikzpicture}
%\hspace{-1cm}
\begin{tikzpicture}
\begin{axis}[ymajorgrids,
grid style=dotted,height=5cm,width=6cm,
          xmin=-1,ymin=0,
          xmax=4,ymax=1.2,
          restrict y to domain=0:1,
          restrict x to domain=-1:4,
          xtick={-1,0,1,2,3,4},
          ytick={0.0,0.2,0.4,0.6,0.8,1.0,1.2},xlabel={$x$},yticklabels={},legend style={draw=none,fill=none,font=\plotfont,anchor=north west, at={(1.1,0.9)}},legend cell align={left},tick label style={font=\plotfont}, label style={font=\plotfont}
          ]
\addplot[col1!0!col2] table [x = {x}, y = {q}] {\etaAe};
\addplot[col1!20!col2] table [x = {x}, y = {q}] {\etaBe};
\addplot[col1!40!col2] table [x = {x}, y = {q}] {\etaCe};
\addplot[col1!60!col2] table [x = {x}, y = {q}] {\etaDe};
\addplot[col1!80!col2] table [x = {x}, y = {q}] {\etaEe};
\addplot[col1!100!col2] table [x = {x}, y = {q}] {\etaFe};

\legend{$\eta = 2^{-1}$,$\eta = 2^{-2}$,$\eta = 2^{-3}$,$\eta = 2^{-4}$,$\eta = 2^{-5}$,$\eta = 2^{-6}$}
\end{axis}
\end{tikzpicture}
    \caption{Plot of the solution \(q\) at time \textbf{left} \(t=1\) and  \textbf{right} \(t=2\). The initial datum is chosen as $q_0 \equiv \chi_{(-1,-0.5)} + \chi_{(1,1.5)}$, the nonlocal kernel in a non-monotone and discontinuous variant as $\gamma \equiv (1-2\cdot) \chi_{(0,0.5)} + \tfrac{1}{2}\chi_{(0.5,2)}$ and the velocity \(V\equiv 1-\cdot\). One can clearly observe a) the violation (compare also \cref{fig:max_q}) the maximum principle for \(\eta=2^{-1}\) and \(t=1\) having in mind that the initial datum satisfies \(\|q_0\|_{\sL^\infty(\R)} \leq 1\) b) the convergence for \(\eta\rightarrow 0\), many of the kinks stemming from the discontinuous kernel.}
    \label{fig:sol}
\end{figure}
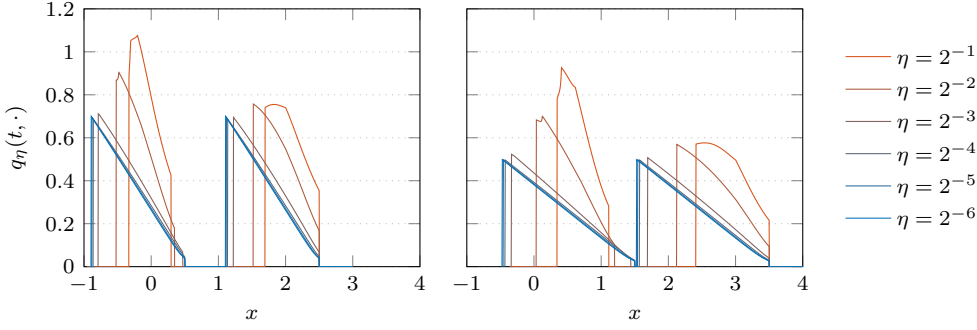

\section{Open problems}\label{sec:open_problems}
We have established a generic result for the singular limit problem in nonlocal (to local) conservation laws for a broad class of nonlocal kernels.
However, several open challenges remain: (\textbf{1}) The case of a constant kernel---which can be considered as the most simplistic nonlocal approximation---remains open.
(\textbf{2}) The kernel looking to one side only (into the direction of \(-\sgn(V')\sgn(q)\)) requires that there be no change in the sign of \(V'\) and initial datum. This restriction is inconvenient, as one would hope also to obtain nonlocal approximations (potentially in a very weak metric) of general local conservation laws (with the problems outlined in \cite{spinolo}). One could try to remedy this problem by considering two-sided kernels, but then there is no maximum principle and classical compactness results in \(\sL^{1}\) for the solution are not applicable (which is in line with \cite{spinolo}). Looking instead (as also done here) to surrogate quantities (like \(E\) in \cref{lem:surrogate_dynamics_E}) and/or weaker types of convergence could prove to be a solution to the raised problem.
(\textbf{3}) In a recent publication \cite{chiarello2023singular}, the singular limit problem for systems of conservation laws has been studied, but the coupling was only on the right-hand side, so  considering non-diagonal systems nonlocal approximations is an appropriate next step.
(\textbf{4}) A singular limit convergence for multi-$d$ nonlocal conservation laws as in \cite{spinola} is entirely open, one reason being that, in such a case, a maximum principle does not hold, and it is difficult to use nonlocal kernels which enable the solution to be expressed solely in terms of the nonlocal operators.
(\textbf{5}) A singular limit convergence on bounded domains \cite{pflug3,Colombo2018,goatin2019well} is open as well, as it is unclear how to obtain \(\sTV\) bounds even in the nonlocal term to satisfy as well the corresponding weak boundary conditions in the sense of \cite{Bardos1979}.
(\textbf{6}) The constant kernel case, i.e., \(\gamma_{\eta}\equiv\tfrac{1}{\eta}\chi_{(0,\eta)}\) resulting in 
\(W[q,\gamma_{\eta}]=\tfrac{1}{\eta}\int_{x}^{x+\eta}q(t,y)\dd y\)
is---for a general initial datum---still open  (though proven in \cite{pflug4} for a monotone initial datum), and our results do not apply to this case (as the power scaling of the constant kernel does not converge to a Dirac distribution).

\section*{Acknowledgments}
Lukas Pflug acknowledges the support of the Collaborative Research Centre 1411 “Design of Particulate Products” (Project-ID 416229255).

\bibliography{sn-biblio}% common bib file

\end{document}